\documentclass[oneside,english]{amsart}
\pdfoutput=1
\usepackage[T1]{fontenc}
\usepackage[latin9]{inputenc}
\usepackage{geometry}
\usepackage{babel}
\usepackage{mathtools}
\usepackage{amsbsy}
\usepackage{amstext}
\usepackage{amsthm}
\usepackage{amssymb}
\usepackage{graphicx}
\usepackage{verbatim}
\usepackage[unicode=true,pdfusetitle, bookmarks=true,bookmarksnumbered=false,bookmarksopen=false, breaklinks=false,backref=false,colorlinks=true] {hyperref}
\usepackage{cleveref}
\usepackage[all]{xy}
\usepackage{xcolor}
\usepackage{caption}
\usepackage{soul}
\usepackage{multirow}
\usepackage{todonotes}
\usepackage{tikz}
\usetikzlibrary{positioning}
\makeatletter

\numberwithin{equation}{section}
\numberwithin{figure}{section}
\theoremstyle{plain}
\newtheorem{thm}{\protect\theoremname}[section]
\newtheorem{cor}[thm]{\protect\corollaryname}
\newtheorem{lem}[thm]{Lemma}
\newtheorem{prop}[thm]{Proposition}
\theoremstyle{definition}
\newtheorem{defn}[thm]{\protect\definitionname}

\newtheorem{qn}{Question}
\newtheorem{rem}[thm]{\protect\remarkname}

\author{Sungkyung Kang}
\address{Center for Geometry and Physics, Institute for Basic Science (IBS), Pohang 37673, Korea}
\email{sungkyung38@icloud.com}

\makeatother

\providecommand{\corollaryname}{Corollary}
\providecommand{\definitionname}{Definition}
\providecommand{\remarkname}{Remark}
\providecommand{\theoremname}{Theorem}

\begin{document}

\title{One stabilization is not enough for contractible 4-manifolds}

\begin{abstract}
We construct an example of a cork that remains exotic after taking a connected sum with $S^2 \times S^2$. Combined with a work of Akbulut-Ruberman, this implies the existence of an exotic pair of contractible 4-manifolds which remains absolutely exotic after taking a connected sum with $S^2 \times S^2$.
\end{abstract}
\maketitle

\section{Introduction}

It is a widely known fact that any pair of homeomorphic simply-connected 4-manifolds are stably diffeomorphic, i.e. diffeomorphic after taking connected sums with finitely many copies of $S^2 \times S^2$, as shown by Wall in \cite{wall}. It is thus a natural question to ask, given such a pair, how many copies of $S^2 \times S^2$ are necessary to make them diffeomorphic? There are lots of families of homeomorphic simply connected 4-manifolds, including simply-connected elliptic surfaces \cite{mandelbaum1979decomposing}, which become diffeomorphic after one stabilization. However, there has been no known example where one stabilization is not enough to obtain a diffeomorphism.

One can try to tackle this problem by constructing a cork which does not trivialize after one stabilization. A triple $(Y,W,f)$ is said to be a \emph{cork} if
\begin{itemize}
    \item $Y$ is a homology 3-sphere,
    \item  $W$ is a contractible 4-manifold with an identification $\partial W = Y$,
    \item $f:Y\rightarrow Y$ is a self-diffeomorphism of $Y$ which does not extend smoothly to $W$; note that $f$ always extends to a self-homeomorphism $W$ by Freedman's theorem \cite{freedman1982topology}.
\end{itemize}

Given a cork $(Y,W,f)$, one can perform a cork twist using the given cork to construct potentially exotic smooth structures on 4-manifolds. In some sense, it is the only way to product exotic 4-manifolds; it is known that, given any finite list of pairwise homeomorphic smooth closed simply-connected 4-manifolds, there exists a single cork from which all manifolds in the given list can be generated by a cork twist \cite{highercorks}. As in the closed case, the phenomenon that all exoticness are killed by sufficiently many stabilizations also occurs in the case of corks, or more generally, 4-manifolds with boundary, as shown in \cite{gompf1984stable}. In particular, given any cork $(Y,W,f)$, there always exists an integer $n>0$ such that $f$ extends smoothly to a self-diffeomorphism of $Y\sharp n(S^2 \times S^2)$.\footnote{Although this statement is not explicitly written in the main theorem of \cite{gompf1984stable}, one can see that this follows easily from Gompf's arguments.} Thus we are naturally led to ask whether ``one is enough'' for corks, i.e. there exists a cork which does not trivialize after one stabilization. To be more precise, we can ask whether there exists a cork $(Y,W,f)$ where $f$ does not smoothly extend to $W\sharp (S^2 \times S^2)$. In this paper, we answer this question in the affirmative, using techniques from involutive Heegaard Floer homology.

\begin{thm}\label{thm:main}
There exists a cork $(Y,W,f)$ such that $f$ does not extend to a self-diffeomorphism of $W\sharp (S^2 \times S^2)$.
\end{thm}

Then, by following the arguments of \cite{akbulut2016absolutely}, one can also prove the existence of an absolutely exotic pair of contractible 4-manifolds which remains absolutely exotic after one stabilization.

\begin{cor}\label{cor:main}
There exist homeomorphic smooth contractible 4-manifolds $W_1,W_2$, with diffeomorphic boundaries, such that $W_1 \sharp (S^2 \times S^2)$ and $W_2 \sharp (S^2 \times S^2)$ are not diffeomorphic.
\end{cor}

Given the above corollary, the next step would be to prove that performing a cork twist by the cork presented in the proof of \Cref{thm:main} produces a pair of closed smooth simply-connected 4-manifolds which are not diffeomorphic after one stabilization. We were unable to prove it, so we leave it as a question.

\begin{qn}
Can we perform a cork twist using the cork we construct in the proof of \Cref{thm:main} to find an exotic closed simply-connected 4-manifold which remains exotic after one stabilization by $S^2 \times S^2$?
\end{qn}

Furthermore, we can also ask the following question.

\begin{qn}\label{qn:qn2}
Is there a cork which does not trivialize after two stabilizations? How about $n$ stabilizations for general $n\ge 2$?
\end{qn}

Note that \Cref{qn:qn2} cannot be answered using the arguments in this paper, as the mapping cone variable $Q$ in involutive Heegaard Floer homology satisfies $Q^2 = 0$.

\subsection*{Organization} In \Cref{sec:prelim}, we present a brief review of Heegaard Floer theoretic facts that we will use in this paper. In \Cref{sec:ansatz}, we give a detailed construction of a cork which we will use to prove \Cref{thm:main}. In \Cref{sec:obstruction}, we describe an algebraic obstruction which we will use to prove \Cref{thm:main}. Finally, in \Cref{sec:proof}, we prove \Cref{thm:main} and \Cref{cor:main}.

\subsection*{Acknowledgement} The author would like to thank Jae Choon Cha, Gary Guth, Kyle Hayden, Jennifer Hom, Seungwon Kim, Robert Lipshitz, Patrick Naylor, JungHwan Park, Mark Powell, Hannah Schwartz, and Ian Zemke for helpful discussions, suggestions, and comments. The author would particularly like to thank Ian Zemke for suggesting the example $2T_{4,5}\sharp -T_{4,9}$ and also for numerous helpful conversations over email. Finally, the author would like to thank an anonymous referee for pointing out gaps in the previous version of the paper.
\section{Heegaard Floer preliminaries} \label{sec:prelim}

\subsection{Heegaard Floer theory, naturality, and functoriality}
Heegaard Floer homology, developed by Ozsv\`{a}th and Sz\`{a}bo in \cite{ozsvath2004holomorphic}, gives a set of invariants which can be used to study 3-manifolds, knots, links, and cobordisms between them. In particular, for each 3-manifold $Y$ together with a $\mathrm{Spin}^c$-structure $\mathfrak{s}$ on $Y$, they associate to it a chain complex $CF^-(Y,\mathfrak{s})$ of $\mathbb{F}_2[U]$-modules, whose homotopy equivalence class depends only on the diffeomorphism class of $(Y,\mathfrak{s})$. Furthermore, when $c_1(\mathfrak{s})$ is torsion, then $CF^-(Y,\mathfrak{s})$ is equipped with an absolute $\mathbb{Q}$-grading, and its graded homotopy equivalence class becomes an invariant of $(Y,\mathfrak{s})$. Furthermore, in this case we also have a localization formula
\[
U^{-1}CF^-(Y,\mathfrak{s})\simeq \mathbb{F}_2[U,U^{-1}].
\]
Note that, when $Y$ is a homology sphere, it carries a unique $\mathrm{Spin}^c$ structure, so we will drop $\mathfrak{s}$ from our notation in this case and just write $CF^-(Y)$ instead.

This theory also has a knot-theoretic counterpart, called knot Floer homology. To a knot $K$ in a $\mathrm{Spin}^c$ 3-manifold $(Y,\mathfrak{s})$, one can associate a chain complex $CFK_{UV}(Y,K,\mathfrak{s})$ of $\mathbb{F}_2[U,V]$-modules whose homotopy equivalence class depends only on the isotopy class of $K$. We will also use its truncations by $V=0$ and $U=V=0$, which we will denote as $CFK^-(Y,K,\mathfrak{s})$ and $\widehat{CFK}(Y,K,\mathfrak{s})$, respectively. When $K$ is rationally null-homologous and $c_1(\mathfrak{s})$ is torsion, $CFK_{UV}(Y,K,\mathfrak{s})$ is equipped with an absolute $\mathbb{Z}$-bigrading, called Alexander and Maslov gradings. In this paper, we will always deal with knots in $S^3$, so we will drop $\mathfrak{s}$ from our notation. Note that, while the construction was first given by Ozsvath and Szabo in \cite{ozsvath2008knot}, we are actually using the formalism established by Zemke in \cite{zemke2019link}.

Although diffeomorphic 3-manifolds (or isotopic knots for $CFK$) induce homotopy equivalent Heegaard Floer chain complexes, the definition of Heegard Floer homology does not take its topological input directly. In fact, given a 3-manifold $Y$, one first represents it as a pointed Heegaard diagram $\mathcal{H}=(\Sigma,\boldsymbol\alpha,\boldsymbol\beta,z)$, and then counts holomorphic disks in a symmetric power of $\Sigma$ with boundary conditions given by $\boldsymbol\alpha$ and $\boldsymbol\beta$ while recording its algebraic intersection numbers with $z$ to define $CF^-(\mathcal{H})$, whose homotopy equivalence class is denoted as 
\[
CF^-(\mathcal{H}) = CF^-(Y) = \bigoplus_{\mathfrak{s} \in \mathrm{Spin}^c(Y)}CF^-(Y,\mathfrak{s}).
\]

While any two Heegaard diagrams representing the same 3-manifold are connected by a sequence of Heegaard moves, and each Heegaard move $\mathcal{H}_1\rightarrow \mathcal{H}_2$ induces a homotopy equivalence $CF^-(\mathcal{H}_1)\rightarrow CF^-(\mathcal{H}_2)$, it is not a priori clear whether every loop of Heegaard moves starting from a diagram $\mathcal{H}$ should induce the identity map on $CF^-(\mathcal{H})$. This problem was resolved up to first order by Juhasz, Thurston, and Zemke \cite{juhasz2021naturality}, by proving that any loop of Heegaard moves indeed induce the identity map up to homotopy. In other words, for any two Heegaard diagrams $\mathcal{H}_1,\mathcal{H}_2$ representing $Y$, the transition map
\[
\Psi_{\mathcal{H}_1\rightarrow \mathcal{H}_2}:CF^-(\mathcal{H}_1)\rightarrow CF^-(\mathcal{H}_2)
\]
is well-defined up to homotopy, and satisfies 
\[
\Psi_{\mathcal{H}\rightarrow\mathcal{H}} \sim \mathrm{id},\quad \Psi_{\mathcal{H}_2\rightarrow \mathcal{H}_3}\circ \Psi_{\mathcal{H}_1\rightarrow \mathcal{H}_2} \sim \Psi_{\mathcal{H}_1\rightarrow \mathcal{H}_3}.
\]
Note that similar statements also hold for knot Floer chain complexes.

Having established the first-order naturality for Heegaard Floer homology, we are now allowed to talk about its (first-order) functoriality under cobordisms. Given two $\mathrm{Spin}^c$ 3-manifolds $Y_1,Y_2$, a 4-dimensional $\mathrm{Spin}^c$-cobordism $(W,\mathfrak{s})$ between them, and a smoothly embedded curve $\gamma \subset W$ connecting the (implicitly chosen) basepoints of $Y_1$ and $Y_2$, Ozsvath and Szabo define in \cite{ozsvath2006holomorphic} a chain map
\[
F^-_{W,\mathfrak{s},\gamma}:CF^-(Y_1,\mathfrak{s}\vert_{Y_1})\rightarrow CF^-(Y_2,\mathfrak{s}\vert_{Y_2}),
\]
whose homotopy class depends only on the smooth isotopy class of $(W,\mathfrak{s},\gamma)$. There is also a knot Floer theoretic counterpart of this cobordism map construction: given knots $K_1\subset Y_1$, $K_2 \subset Y_2$, a smoothly embedded oriented surface $S\subset W$ satisfying $\partial S = K_1 \cup K_2$, together with a suitable decoration (see \cite[Section 1]{zemke2019link} for a definition of decorated cobordisms), Zemke defines a chain map 
\[
F_{W,S,\mathfrak{s}}:CFK_{UV}(Y_1,K_1)\rightarrow CFK_{UV}(Y_2,K_2)
\]
whose homotopy class depends only on the smooth isotopy class of $(W,S,\mathfrak{s})$. Note that these cobordism maps are functorial under compositions of cobordisms (with extra data).

\begin{rem} \label{rem:invariantcW}
Functoriality of Heegaard Floer homology can be used to construct a diffeomorphism (rel boundary) invariant of contractible 4-manifolds as follows. Given a contractible 4-manifold $W$ bounding a homology sphere $Y$, we can consider $W$ as a cobordism from $S^3$ to $Y$ by removing a small open ball in its interior. Choose basepoints on $S^3$ and $Y$, one on each, and a smoothly embedded curve $\gamma \subset W$ which connects the basepoints. Then we have a cobordism map
\[
F^-_W = F^-_{W,\mathfrak{s},\gamma}:CF^-(S^3)\rightarrow CF^-(Y),
\]
which is well-defined up to homotopy. Since $CF^-(S^3)\simeq \mathbb{F}_2[U]$, the homology class 
\[
[F^-_W(1)]\in HF^-(Y) = H_\ast(CF^-(Y))
\]
is well-defined; we will denote this class as $c_W$. 
\end{rem}
\begin{rem} \label{rem:invarianttD}
One can also use cobordism maps to define an invariant of smooth slice disks. For example, if a knot $K$ bounds a smooth disk $D$ in $B^4$, then endowing $D$ with the simplest possible decoration induces a cobordism map
\[
F_{D}:\widehat{CFK}(S^3,\mathrm{unknot})\rightarrow \widehat{CFK}(S^3,K);
\]
note that we are using the hat-flavored version, $\widehat{CFK}$, instead of $CFK_{UV}$, since that's what we will use in this paper. Since $\widehat{CFK}(S^3,\mathrm{unknot})\simeq \mathbb{F}_2$, we denote the class $[F_{D}(1)] \in \widehat{HFK}(S^3,K)=H_\ast(\widehat{CFK}(S^3,K))$ by $t_D$, which was defined first in \cite{juhaszmarengon}.
\end{rem}

\subsection{Involutive Heegaard Floer theory}

Naturality can also be used to extract more data from Heegaard Floer homology. In particular, Hendricks and Manolescu defined involutive Heegaard Floer homology in \cite{hendricks2017involutive} as follows. Given an oriented 3-manifold $Y$, choose a pointed Heegaard diagram $\mathcal{H}=(\Sigma,\boldsymbol\alpha,\boldsymbol\beta,z)$ representing $Y$, and consider its conjugate diagram 
\[
\bar{\mathcal{H}}=(-\Sigma,\boldsymbol\beta,\boldsymbol\alpha,z)
\]
which also represents $Y$. We have a canonical identification map
\[
\mathrm{id}:CF^-(\mathcal{H})\rightarrow CF^-(\bar{\mathcal{H}}),
\]
and composing it with $\Psi_{\bar{\mathcal{H}}\rightarrow \mathcal{H}}$ gives the involutive action 
\[
\iota_Y = \Psi_{\bar{\mathcal{H}}\rightarrow \mathcal{H}}\circ \mathrm{id},
\]
which is a homotopy autoequivalence, well-defined up to homotopy, satisfying $\iota_Y^2 \sim \mathrm{id}$. Note that, since $\iota_Y$ acts on the set of $\mathrm{Spin}^c$-structures on $Y$ by conjugation, if we are given a self-conjugate $\mathrm{Spin}^c$-structure $\mathfrak{s}$ on $Y$, the action of $\iota_Y$ restricts to $CF^-(Y,\mathfrak{s})$. Also, if we consider the involutive Heegaard Floer chain complex 
\[
CFI^-(Y) = \mathrm{Cone}(1+\iota_Y:CF^-(Y)\rightarrow CF^-(Y)),
\]
i.e. $CFI^-(Y) = CF^-(Y)\otimes \mathbb{F}_2[U,Q]/(Q^2)$ together with the mapping cone differential 
\[
\partial_I = \partial + Q(1+\iota_Y)
\]
where $\partial$ denotes the differential of $CF^-(Y)$, then we have a splitting 
\[
CFI^- (Y) = \bigoplus_{\{\mathfrak{s},\bar{\mathfrak{s}}\}\subset \mathrm{Spin}^c(Y)} CFI^-(Y,\{\mathfrak{s},\bar{\mathfrak{s}}\}).
\]
along the conjugation-orbits of $\mathrm{Spin}^c$-structures on $Y$. Here, $\bar{\mathfrak{s}}$ denotes the conjugate of $\mathfrak{s}$.

One can also consider a similar construction regarding knot Floer homology, which again gives an action well-defined up to homotopy. In particular, given a knot $K$ inside a 3-manifold $Y$ with a self-conjugate $\mathrm{Spin}^c$-structure $\mathfrak{s}$, we can define a homotopy skew-autoequivalence $\iota_K$ on $CFK_{UV}(Y,K,\mathfrak{s})$, satisfying $\iota^2_K\sim \xi_K$. Here, $\xi_K$ denotes the Sarkar involution, defined in \cite{sarkar2015moving}; since it squares to identity up to homotopy, we have $\iota^4_K \sim \mathrm{id}$. Note that $\iota_K$ is a skew-autoequivalence because its action intertwines the actions of $U$ and $V$ with the actions of $V$ and $U$, respectively.

Since the naturality and functoriality of Heegaard Floer homology is known only up to first order, it is a priori unclear whether involutive Heegaard Floer homology should also be first-order natural and functorial. Fortunately, those results were proven by Hendricks, Hom, Stoffregen, and Zemke in \cite{hendricks2022naturality}. In particular, its transition maps are also well-defined up to homotopy, and given cobordism data $(W,\mathfrak{s},\gamma)$ as before, where $\mathfrak{s}$ is now assumed to be self-conjugate, we have an associated cobordism map 
\[
F^I_{W,\mathfrak{s},\gamma}:CFI^-(Y_1,\mathfrak{s}\vert_{Y_1})\rightarrow CFI^-(Y_2,\mathfrak{s}\vert_{Y_2}),
\]
whose homotopy class is determined by the smooth isotopy class of $(W,\mathfrak{s},\gamma)$.

\subsection{Involutive bordered Floer homology}
Bordered Floer theory is a version of Heegaard Floer theory for 3-manifolds with boundary. In particular, given an oriented 3-manifold $Y$ with a connected parametrized boundary $\partial Y$ (such manifolds are called bordered 3-manifolds), one associates to it a type-D structure $\widehat{CFD}(Y)$ and a type-A structure $\widehat{CFA}(Y)$. Furthermore, if $\partial Y$ has two connected components, then one can associate to it a type DA structure $\widehat{CFDA}(Y)$; note that we can also consider type DD and type AA structures, but we will not consider them here.

Bordered Floer homology is useful when computing Heegaard Floer homology of glued manifolds. Given two bordered 3-manifolds $Y_1,Y_2$ with a prescribed identification $-\partial Y_1 = \partial Y_2$, we have a pairing formula 
\[
\widehat{CFA}(Y_1)\boxtimes \widehat{CFD}(Y_2) \simeq \widehat{CF}(Y_1 \cup Y_2),
\]
where $\boxtimes$ denotes the box-tensor product, which gives a pairing of a type-D structure with a type-A structure and produces a chain complex as an output. There are lots of versions of pairing formulae; for example, when $Y_1$ has instead two boundary components, say $\partial_1 Y_1$ and $\partial_2 Y_1$, and we are given an identification $-\partial_2 Y_1 = \partial Y_2$, then we have 
\[
\widehat{CFD}(Y_1 \cup Y_2) \simeq \widehat{CFDA}(Y_1) \boxtimes \widehat{CFD}(Y_2).
\]
In general, pairing a ``type D'' boundary component of one bordered 3-manifold with a ``type A'' boundary of another bordered 3-manifold gives a pairing formula.

Recall that, when we have a knot $K$ inside a closed connected oriented 3-manifold $Y$, we can define its knot Floer chain complex, $CFK_{UV}(Y,K)$. Similarly we can consider the case when we have a knot $K$ inside a bordered 3-manifold $Y$, where $\partial Y$ is connected. Here, our choice of basepoints on $K$ gains importance: as in knot Floer theory for knots in closed 3-manifolds, we have two basepoints $z,w$ on $K$, but now we require that $z\in \partial Y$. In this setting, we can define a type-D structure $CFD^-(Y,K)$ and a type-A structure $CFA^-(Y,K)$ over the ring $\mathbb{F}_2[U]$. When we have another bordered 3-manifold $Z$ with a boundary identification $-\partial Y = \partial Z$, then we have gluing formulae 
\[
\begin{split}
    \widehat{CFA}(Z)\boxtimes CFD^-(Y,K) &\simeq CFK^-(Y\cup Z,K), \\
    CFA^-(Y,K)\boxtimes \widehat{CFA}(Z) &\simeq CFK^-(Y\cup Z,K),
\end{split}
\]
where $CFK^-(Y\cup Z,K)$ denotes the truncation of $CFK_{UV}(Y\cup Z,K)$ where the formal variables associated to the basepoint $z$ and $w$ are $0$ and $U$, respectively. Note that we can also consider $\widehat{CFD}(Y,K)$ and $\widehat{CFA}(Y,K)$, which are the truncations of $CFD^-(Y,K)$ and $CFA^-(Y,K)$ by $U=0$; they are related by $\widehat{CFK}(Y\cup Z,K)$ via pairing formulae.

As in the cases of Heegaard Floer theory and knot Floer theory, one can also consider involutive actions in bordered Floer theory. However, the action is not a homotopy autoequivalence anymore. Instead, as defined by Hendricks and Lipshitz \cite{hendricks2019involutivebordered}, it takes the form
\[
\begin{split}
\iota_M &: \widehat{CFDA}(\mathbf{AZ})\boxtimes \widehat{CFD}(M)\rightarrow\widehat{CFD}(M), \\
\iota_M &: \widehat{CFA}(M) \boxtimes \widehat{CFDA}(\overline{\mathbf{AZ}}) \rightarrow \widehat{CFD}(M),
\end{split}
\]
when $M$ is a bordered 3-manifold with one torus boundary, where $\mathbf{AZ}$ denotes the genus 1 Auroux-Zarev piece and $\overline{\mathbf{AZ}}$ is its reverse; see \cite{lipshitz2018bordered} for the definition of $\mathbf{AZ}$. The bordered involutions are related to the involution action on Heegaard Floer theory via a gluing formula; in particular, the map
\begin{equation} \label{eqn:involutive}
\begin{split}
\widehat{CF}(M_1 \cup M_2) &\xrightarrow{\text{pairing}} \widehat{CFA}(M_1) \boxtimes \widehat{CFD}(M_2) \\
&\xrightarrow{\simeq} \widehat{CFA}(M_1)\boxtimes \widehat{CFDA}(\overline{\mathbf{AZ}})\boxtimes \widehat{CFDA}(\mathbf{AZ})\boxtimes \widehat{CFD}(M_2) \\
&\xrightarrow{\iota_1 \boxtimes \iota_2} \widehat{CFA}(M_1) \boxtimes \widehat{CFD}(M_2) \xrightarrow{\text{pairing}} \widehat{CF}(M_1 \cup M_2)
\end{split}
\end{equation}
is homotopic to the involution $\iota_{M_1 \cup M_2}$ on $\widehat{CF}(M_1 \cup M_2)$, as shown in \cite{lipshitz2018bordered}.

Given a knot $K$ in $S^3$, we can also recover the action of $\iota_K$ on $\widehat{CFK}(S^3,K)$ via involutive bordered Floer theory. Consider the 0-framed knot complement $S^3 \setminus K$, which admits an involution 
\[
\iota_{S^3 \setminus K}:\widehat{CFA}(S^3 \setminus K) \boxtimes \widehat{CFDA}(\overline{\mathbf{AZ}})\rightarrow \widehat{CFA}(S^3 \setminus K).
\]
By considering the longitudinal knot $\nu$ inside the $\infty$-framed solid torus $T_\infty$, we also get a type-A structure $\widehat{CFA}(T_\infty,\nu)$. Then, as shown in \cite{kang2022involutive}, there exists a type-D morphism 
\[
f:\widehat{CFD}(T_\infty,\nu)\rightarrow \widehat{CFDA}(\mathbf{AZ})\boxtimes \widehat{CFD}(T_\infty,\nu)
\]
such that the map 
\[
\begin{split}
    \widehat{CFK}(S^3,K) &\simeq \widehat{CFA}(S^3 \setminus K)\boxtimes \widehat{CFD}(T_\infty,\nu) \\ 
    &\xrightarrow{\iota_{S^3 \setminus K}^{-1} \boxtimes f} \widehat{CFA}(S^3 \setminus K)\boxtimes \widehat{CFDA}(\overline{\mathbf{AZ}})\boxtimes \widehat{CFDA}(\overline{\mathbf{AZ}})\boxtimes \widehat{CFD}(T_\infty,\nu) \\
    &\xrightarrow{\simeq} \widehat{CFA}(S^3 \setminus K)\boxtimes \widehat{CFD}(T_\infty,\nu) \simeq \widehat{CFK}(S^3,K)
\end{split}
\]
is homotopic to the truncation of either $\iota_K$ or its homotopy inverse $\iota^{-1}_K$.

\begin{rem}
By the bordered naturality \cite[Theorems 2.4, 2.5, and 2.8]{guth2024invariant}, the action of $\iota_M$ for a bordered 3-manifold $M$ is uniquely determined up to homotopy. However, all arguments in this paper can be modified to make sense even without naturality, as we still have naturality for closed manifolds (so we always have naturality ``after gluing''), as shown in \cite{juhasz2021naturality}.
\end{rem}

\subsection{From summands of $CFK_\mathcal{R}$ to summands of type-D for $\mathcal{R}$-multirectangular knots} \label{subsec:borderedcomplement}
Given a knot $K$ in $S^3$, the Lipshitz-Ozsv\'{a}th-Thurston correspondence \cite[Theorem 11.26]{lipshitz2018bordered} provides a purely algebraic way to compute $\widehat{CFD}(S^3 \setminus K)$ (we always use the Seifert framing, i.e. 0-framing for knot complements) from the knot Floer chain complex of $K$. Since it is purely algebraic, it can be defined for any finitely generated free chain complex over $\mathcal{R} = \mathbb{F}_2[U,V]/(UV)$. We write this correspondence as follows:
\[
\text{chain complex } C \text{ over }\mathcal{R} \mapsto \text{type }D \text{ structure }\mathcal{M}_C.
\]
Up to homotopy equivalence, this correspondence is additive with respect to direct sum, i.e.
\[
\mathcal{M}_{C_1 \oplus \cdots \oplus C_n} \simeq \mathcal{M}_{C_1} \oplus \cdots \oplus \mathcal{M}_{C_n}.
\]

The problem with this correspondence is that, while it is purely algebraic and thus explicitly computable, it is unclear why it should be compatible with respect to various symplectically defined maps. Because of this problem, we would have to work with a completely different way to construct splittings of $\widehat{CFD}(S^3 \setminus K)$ from splittings of $CFK_\mathcal{R}(S^3,K)$. This was done in a joint work of the author with Guth \cite{guth2024invariant}. We will start by summarizing some of its constructions and results. From now on, we will use the bordered naturality \cite[Theorems 2.4, 2.5, and 2.8]{guth2024invariant} from that work to make statements and arguments more concise, although it is not strictly necessary.

We first recall the construction of a chain map $\Lambda_K$ from the endomorphism space of $\widehat{CFD}(S^3,K)$ to the endomorphism space of $CFK_\mathcal{R}(S^3,K)$. Given a degree-preserving type D endomorphism $f:\widehat{CFD}(S^3 \setminus K)\rightarrow \widehat{CFD}(S^3 \setminus K)$, we use the morphism pairing theorem \cite[Theorem 1]{lipshitz2011heegaard} to represent its homotopy class as a hat-flavored Heegaard Floer homology class:
\[
[f]\in \widehat{HF}(S^3 _0 (K\sharp -K),[0]).
\]
This is because we have an identification
\[
-(S^3 \setminus K) \cup (S^3 \setminus K) \simeq S^3 _0 (K\sharp -K).
\]
Then we choose a very large integer $N$ and take the zero spin structure $\mathfrak{s}_0$ of the lens space $L(N,1)$. Considering the cobordism map induced by the canonical 2-handle cobordism $W_{K,N}$ from $S^3 _0 (K\sharp -K)\sharp L(N,1)$ to $S^3 _N (K\sharp -K)$, together with the ``zero'' spin structure $\mathfrak{s}_0$ on $W_{K,N}$ which restricts to the connected sum of the zero spin structure $[0]$ on $S^3 _0 (K\sharp -K)$ and $\mathfrak{s}_0$ on $L(N,1)$, we get the following map:
\[
\widehat{HF}(S^3 _0 (K\sharp -K),[0])\xrightarrow{-\otimes \mathrm{generator}}  \widehat{HF}(S^3 _0 (K\sharp -K),[0]) \otimes \widehat{HF}(L(N,1),\mathfrak{s}_0) \xrightarrow{F_{W,\mathfrak{s}_0}} \widehat{HF}(S^3 _N (K\sharp -K),[0]).
\]
Then we compose it with the large surgery isomorphism:
\[
\Gamma_{K,0}:\widehat{HF}(S^3 _N (K\sharp -K),[0]) \xrightarrow{\simeq} \hat{A}_0(K) \subset HFK_\mathcal{R}(S^3,K\sharp -K).
\]
Finally, we take the map induced on homology by the following homotopy equivalence:
\[
CFK_\mathcal{R}(S^3,K\sharp -K)\simeq \mathrm{End}_\mathcal{R}(CFK_\mathcal{R}(S^3,K)).
\]
Composing everything, we get the desired map (in homology, for simplicity; the same map can in fact be defined in chain level)
\[
\Lambda:H_\ast(\mathrm{End}(\widehat{CFD}(S^3 \setminus K)))\rightarrow H_\ast(\mathrm{End}_\mathcal{R}(S^3,K)).
\]
We will use the following three theorems from \cite{guth2024invariant}, which are all proven using the cobordism map interpretation of composition maps in bordered Floer homology \cite[Theorem 1.1]{cohen2023composition}.

\begin{thm}[{\cite[Corollary 3.14]{guth2024invariant}}] \label{thm:cited_thm1_involution}
    $\Lambda$ homotopy-commutes with conjugations by $\iota_K$ on its codomain and $\iota_{S^3 \setminus K}$ on its domain.
\end{thm}
\begin{thm}[{\cite[Proposition 3.1 and the claim in its proof]{guth2024invariant}}] \label{thm:cited_thm2_projectioncorrespondence}
    $\Lambda$ induces a bijection between homotopy classes of degree(or bidegree)-preserving endomorphisms, and especially, homotopy classes of projections.
\end{thm}
\begin{thm}[{\cite[Proposition 3.2]{guth2024invariant}}; stated here in terms of hat-flavored HFK\footnote{The original statement in \cite{guth2024invariant} uses the bordered diagram $\mathbb{X}$, which represents the longitudinal knot, together with a free basepoint (on the boundary) inside an infinity-framed solid torus. Here, our statement instead uses $(T_\infty,\nu)$ without extra free basepoint, but the proof is identical, and thus we state this theorem without proof.}] \label{thm:cited_thm3_hatflavor_recovery}
    For any (degree-preserving) endomorphism $f:\widehat{CFD}(S^3 \setminus K)\rightarrow \widehat{CFD}(S^3 \setminus K)$, the box-tensored map
    \[
    \mathrm{id}\boxtimes f:\widehat{CFA}(T_\infty,\nu)\boxtimes \widehat{CFD}(S^3 \setminus K) \rightarrow \widehat{CFA}(T_\infty,\nu)\boxtimes \widehat{CFD}(S^3 \setminus K),
    \]
    where $T_\infty$ denotes the infinity-framed solid torus and $\nu$ denotes its longitudinal knot (with one basepoint on $\partial T_\infty$), is homotopic to the hat-flavored truncation (i.e. $U=V=0$) of $\Lambda(f)$ under the identification
    \[
    \widehat{CFK}(S^3,K) \simeq \widehat{CFA}(T_\infty,\nu)\boxtimes \widehat{CFD}(S^3 \setminus K)
    \]
    given by the pairing theorem \cite[Theorem 11.19]{lipshitz2018bordered}.
\end{thm}
\begin{rem}
    When applying \Cref{thm:cited_thm2_projectioncorrespondence}, we will often talk about kernels and images of projection maps on chain complexes and type D modules, as they give direct summands. Given a projection morphism (chain maps for chain complexes and type D morphisms for type D modules) $p,p^\prime$ that are homotopic to each other, we will often use the fact that they images are homotopy equivalent, and kernels are also homotopy equivalent. This is because $\ker(p)$ is homotopy equivalent to the mapping cone of $p$ and $\mathrm{Im}(p)$ is the kernel of $\mathrm{id}+p$.
\end{rem}

It follows from \Cref{thm:cited_thm2_projectioncorrespondence} that splittings of $\widehat{CFD}(S^3 \setminus K)$ are in bijective correspondence up to homotopy with splittings of $CFK_\mathcal{R}(S^3,K)$; then it follows from \Cref{thm:cited_thm1_involution} that this ``splitting correspondence'' maps $\iota_{S^3 \setminus K}$-invariant splittings to $\iota_K$-invariant splittings and vice versa. A priori, it is unclear whether this splitting correspondence is compatible with the one given by the Lipshitz-Ozsv\'{a}th-Thurston correspondence; this turns out to be true, at least up to homotopy, and was proven throughout \cite[Section 4]{guth2024invariant}, but the proof is quite complicated. So, in this paper, we will not rely on that part of \cite{guth2024invariant} and instead develop a way to work around it. The price we pay is that we can only work with the one particular knot that we discuss throughout the paper.

We will put down a very restrictive condition on the knots that we will use. 
\begin{defn}
    A chain complex $C$ over $\mathcal{R}$ is \emph{rectangular} if it admits a model with four free generators $c_1,c_2,c_3,c_4$, where the differential is given by
    \[
    \partial c_1 = U^i c_2 + V^j c_3,\quad \partial c_2 = V^j c_4,\quad \partial c_3 = U^i c_4,\quad \partial c_4 = 0.
    \]
    Also, if $C$ admits a direct summand $D$ (up to homotopy equivalence), we say that $D$ is a \emph{free summand} (of $C$) if it is homotopy equivalent to a free chain complex with one generator and zero differential. We say that $C$ is \emph{simply $\mathcal{R}$-multirectangular} if it is homotopy equivalent to the direct sum of several rectangular complexes, and \emph{$\mathcal{R}$-multirectangular} if it is homotopy equivalent to the direct sum of one free summand and several rectangular complexes. 

    Finally, we say that a knot $K$ is $\mathcal{R}$-multirectangular if $CFK_\mathcal{R}(S^3,K)$ is $\mathcal{R}$-multirectangular.
\end{defn}
A nice property of $\mathcal{R}$-multirectangular complexes is that some of its direct summands can be distinguished from each other only by looking at the hat-flavored truncations.
\begin{defn}
    Given a chain complex $C$ of $\mathcal{R}$-modules, we define its \emph{hat-flavored truncation} $\hat{C}$ as the quotient complex
    \[
    \hat{C}= C\otimes_{\mathcal{R}} \mathbb{F}_2,
    \]
    where $\mathbb{F}_2$ is regarded as an $\mathcal{R}$-module via the identification $\mathbb{F}_2 \simeq \mathcal{R}/(U,V)$.
\end{defn}

\begin{lem} \label{lem: complement is free}
    Let $L,M,N$ be finitely generated chain complex of $\mathcal{R}$-modules satisfying $M=N\oplus L$. Suppose that $M,N$ are free. Then $L$ is homotopy equivalent to a finitely generated free chain complex of $\mathcal{R}$-modules.
\end{lem}
\begin{proof}
    $L$ is homotopy equivalent to the mapping cone of the inclusion $N\subset M$. Since $N$ and $M$ are both free, this mapping cone is also free.
\end{proof}

\begin{lem} \label{lem: summand of multirect is multirect}
    Let $C$ be an $\mathcal{R}$-multirectangular complex and $D$ be its direct summand. Then $D$ is either $\mathcal{R}$-multirectangular or simply $\mathcal{R}$-multirectangular.
\end{lem}
\begin{proof}
    Write $C \simeq D\oplus D^\prime$. By \Cref{lem: complement is free}, we may assume without loss of generality that $D^\prime$ is  finitely generated and free over $\mathcal{R}$. Then we may use \cite[Corollary 4.2]{popovic2023link} to uniquely decompose $D$ and $D^\prime$, up to rearrangement and homotopy equivalence, as follows:
    \[
    D\simeq S_1 \oplus \cdots \oplus S_n \oplus L_1 \oplus \cdots \oplus L_m,\quad D^\prime\simeq S^\prime_1 \oplus \cdots \oplus S^\prime_{n^\prime} \oplus L^\prime_1 \oplus \cdots \oplus L^\prime_{m^\prime},
    \]
    where each $S_i,S^\prime_i$ are snake complexes and $L_j,L^\prime_j$ are local systems. Hence we get
    \[
    C \simeq \left( \bigoplus_i S_i  \right) \oplus \left( \bigoplus_i S^\prime_i  \right) \oplus \left( \bigoplus_i L_i  \right) \oplus \left( \bigoplus_i L^\prime_i  \right).
    \]
    However, since $C$ is $\mathcal{R}$-multirectangular, we already know that it decomposes as
    \[
    C \simeq \mathcal{R} \oplus (\text{rectangular complexes}).
    \]
    Since rectangular complexes are local systems and $\mathcal{R}$ is a snake complex, it follows from the uniqueness part of \cite[Corollary 4.2]{popovic2023link} that 
    \begin{itemize}
        \item either $n=0$, $n^\prime=1$, and $S^\prime_1 \simeq \mathcal{R}$, 
        \item or $n=1$, $n^\prime=0$, and $S_1 \simeq \mathcal{R}$,
    \end{itemize}
    and all $L_i$ and $L^\prime_i$ are rectangular. Therefore $D$ (and also $D^\prime$) are either $\mathcal{R}$-multirectangular or simply $\mathcal{R}$-multirectangular.
\end{proof}

\begin{lem} \label{lem: hat acyclic implies acyclic}
    Let $C$ be a finitely generated free chain complex of $\mathcal{R}$-modules. If the hat-flavored truncation of $C$ is acyclic, then $C$ is acyclic.
\end{lem}
\begin{proof}
    We use \cite[Corollary 4.2]{popovic2023link} to write
    \[
    C \simeq S_1 \oplus \cdots \oplus S_n \oplus L_1 \oplus \cdots \oplus L_m
    \]
    for snake complexes $S_i$ and local systems $L_j$. From the definitions of snake complexes and local systems, it is clear that their hat-flavored truncations are never acyclic. Therefore we deduce that $C$ is acyclic.
\end{proof}

\begin{lem} \label{lem:hat-flavor-separation}
    Let $C$ be a (bigraded) free $\mathcal{R}$-multirectangular complex and $D_1,D_2$ be its free direct summands. Suppose that $H_\ast(\hat{D}_i)$ is 5-dimensional vector space over $\mathbb{F}_2$ for each $i$. If $H_\ast(\hat{D}_1)$ is (bigraded) isomorphic to $H_\ast(\hat{D}_2)$, then $D_1$ is (bigraded) $\mathcal{R}$-linearly homotopy equivalent to $D_2$.
\end{lem}
\begin{proof}
    We know from \Cref{lem: summand of multirect is multirect} that $D_1$ and $D_2$ are both either $\mathcal{R}$-multirectangular or simply $\mathcal{R}$-multirectangular. If it is simply $\mathcal{R}$-multirectangular, then since the hat-flavored truncation of any rectangular complex has 4-dimensional homology, we deduce that $H_\ast(\hat{D}_i)$ should be a multiple of 4, a contradiction. Hence we see that each $D_i$ is $\mathcal{R}$-multirectangular. Furthermore, the same reasoning shows that it cannot have more than one rectangular summand, so we get 
    \[
    D_i \simeq F_i \oplus R_i,
    \]
    where $F_i$ is free and $R_i$ is rectangular.

    Now we consider the \emph{$U$-localizations} of $D_i$; given a chain complex $C$ over $\mathcal{R}$, we define its $U$-localization as the chain complex over $\mathbb{F}_2[U,U^{-1}]$ that we get from $C$ by first truncating $C$ by $V=0$ and then formally inverting $U$. It is clear that the $U$-localization of a free summand gives $\mathbb{F}_2[U,U^{-1}]$ (with zero differential) up to homotopy equivalence and the $U$-localization of a rectangular summand is acyclic. Hence the homology of $U$-localizations of $C,D_1,D_2$ are free of rank 1 over $\mathbb{F}_2[U,U^{-1}]$. However, since $D_1,D_2$ are direct summands of $C$, we see that they should be all (bigraded) isomorphic. This means that $U$-localizations of $F_1$ and $F_2$ are (bigraded) homotopy equivalent over $\mathbb{F}_2[U,U^{-1}$. However, it is clear that homotopy equivalence classes of $F_i$ are determined by the bidegree of the generator of their homology, so we deduce that
    \[
    F_1 \simeq F_2,
    \]
    and hence we also get $H_\ast(\hat{R}_1)\simeq H_\ast(\hat{R}_2)$. However, since $R_1$ and $R_2$ are rectangular, their homotopy equivalence classes are determined by the homology of their hat-flavored truncations. Thus we have
    \[
    R_1 \simeq R_2,
    \]
    and therefore $D_1 \simeq D_2$, as desired.
\end{proof}

\begin{cor} \label{cor: technicalcor}
    Let $C$ be a (bigraded) free $\mathcal{R}$-multirectangular complex and $D_1,D_2,E_1,E_2$ be direct summands of $C$ satisfying
    \[
    C \simeq D_1 \oplus E_1 \simeq D_2 \oplus E_2.
    \]
    Suppose that $H_\ast(\hat{D}_1)\simeq H_\ast(\hat{D}_2)$ and they are 5-dimensional over $\mathbb{F}_2$. Then $D_1 \simeq D_2$ and $E_1 \simeq E_2$.
\end{cor}
\begin{proof}
    The first part, i.e. $D_1 \simeq D_2$, is just \Cref{lem:hat-flavor-separation}. Then $E_1 \simeq E_2$ follows from \Cref{lem: complement is free} and the unique decomposability up to homotopy equivalence of free (bigraded) finitely generated chain complexes over $\mathcal{R}$ \cite[Corollary 1.2]{popovic2023link}.
\end{proof}

This corollary allows us to prove the following proposition, which will be very useful later on.

\begin{prop} \label{prop:keyprop}
    Let $K$ be an $\mathcal{R}$-multirectangular knot. Suppose that $CFK_\mathcal{R}(S^3,K)$ admits $\iota_K$-invariant direct summands $C,D$ such that $H_\ast(\hat{C})$ is 5-dimensional over $\mathbb{F}_2$. Consider the projection endomorphism $p$ which is identity on $C$ and zero on $D$. Let $p^\prime$ denote the projection of $\widehat{CFD}(S^3 \setminus K)$ satisfying $\Lambda(p^\prime) \sim p$, which is defined uniquely up to homotopy by \Cref{thm:cited_thm2_projectioncorrespondence}. Then $M=\mathrm{Im}(p^\prime)$ and $N=\ker(p^\prime)$ are $\iota_{S^3 \setminus K}$-invariant (up to homotopy) direct summands of $\widehat{CFD}(S^3 \setminus K)$ satisfying $M\simeq \mathcal{M}_C$ and $N\simeq \mathcal{M}_D$.
\end{prop}
\begin{proof}
    The $\iota_{S^3\setminus K}$-invariance follows directly from \Cref{thm:cited_thm1_involution}; we only have to prove that $M\simeq \mathcal{M}_C$ and $N\simeq \mathcal{M}_D$. To show this, we start by observing that, by the additivity of Lipshitz-Ozsv\'{a}th-Thurston correspondence, we have a homotopy equivalence
    \[
    \widehat{CFD}(S^3 \setminus K)\simeq \mathcal{M}_{CFK_\mathcal{R}(S^3,K)} \simeq \mathcal{M}_C \oplus \mathcal{M}_D.
    \]
    Let $\tilde{p}$ denote the projection of $\widehat{CFD}(S^3 \setminus K)$ that is identity on $\mathcal{M}_C$ and zero on $\mathcal{M}_D$. By \Cref{thm:cited_thm2_projectioncorrespondence}, we have a projection $\Lambda(\tilde{p})$ on $CFK_\mathcal{R}(S^3,K)$, unique up to homotopy. Denote its image as $\tilde{C}$ and kernel by $\tilde{D}$. Then we have two splittings
    \[
    CFK_\mathcal{R}(S^3,K) \simeq C\oplus D \simeq \tilde{C}\oplus \tilde{D}.
    \]
    By \Cref{thm:cited_thm3_hatflavor_recovery}, we see that
    \[
    H\ast(\hat{\tilde{C}})\simeq \widehat{CFA}(T_\infty,\nu)\boxtimes \mathcal{M}_C.
    \]
    Since $\widehat{CFA}(T_\infty,\nu)$ admits a model with one $\iota_1$-generator, no $\iota_0$-generator, zero differential, and no higher $A_\infty$-operations. we see that this box tensor product is just $H_\ast(\hat{C})$, i.e. we have
    \[
    H_\ast(\hat{\tilde{C}})\simeq H_\ast(\hat{C}).
    \]
    Since we assumed that $H_\ast(\hat{C})$ is 5-dimensional, we can apply \Cref{cor: technicalcor} to deduce that
    \[
    C\simeq \tilde{C},\quad D\simeq \tilde{D}.
    \]
    Hence we can consider the following homotopy autoequivalence:
    \[
    f:CFK_\mathcal{R}(S^3,K)\xrightarrow{\simeq} C\oplus D \xrightarrow{\simeq} \tilde{C}\oplus \tilde{D} \xrightarrow{\simeq} CFK_\mathcal{R}(S^3,K).
    \]
    By construction, $f$ satisfies $f \circ p \circ f^{-1} \sim \Lambda(\tilde{p})$. Since its left and right hand sides are both projections, it follows from \Cref{thm:cited_thm2_projectioncorrespondence} that $\Lambda^{-1}(f)$ is a homotopy autoequivalence of $\widehat{CFD}(S^3 \setminus K)$ that is well-defined up to homotopy, and 
    \[
    \Lambda^{-1}(f)^{-1} \circ p^\prime \circ \Lambda^{-1}(f) \sim \tilde{p}.
    \]

    Now, by taking a (finitely generated and free) reduced model of $\widehat{CFD}(S^3 \setminus K)$, so that it has no acyclic type D direct summands, we may homotope the homotopy equivalence $\Lambda^{-1}(f)$ to an automorphism. To see why, choose any chain map representative of the homotopy class $\Lambda^{-1}(f)$, which we again write as $\Lambda^{-1}(f)$, and consider $(\Lambda^{-1}(f))^N$ for very large postive integers $N$; by the finite generation of $\widehat{CFD}(S^3 \setminus K)$, as $N\rightarrow \infty$, $(\Lambda^{-1}(f))^N$ should stabilize to an automorphism of some type D direct summand $X$ of $\widehat{CFD}(S^3 \setminus K)$ whose inclusion is a homotopy equivalence. So we have a splitting
    \[
    \widehat{CFD}(S^3 \setminus K) \simeq X\oplus Y
    \]
    where $Y$ is acyclic. But we have taken a reduced model of $\widehat{CFD}(S^3 \setminus K)$ so that it has no acyclic direct summand. Hence $Y=0$, which implies that any sufficiently large power of $\Lambda^{-1}(f)$ (and thus $\Lambda^{-1}(f)$ itself) should have zero kernel. By applying the finite generation condition again, we see that $\Lambda^{-1}(f)$ is in fact an automorphism of $\widehat{CFD}(S^3 \setminus K)$.

    Hence we see that $p^\prime$ is conjugate to $p$ via (type D) automorphisms. This implies that their images and kernels are isomorphic type D structures. Therefore we get
    \[
    M \simeq \mathrm{Im}(p^\prime) \simeq \mathrm{Im}(\tilde{p})\simeq \mathcal{M}_C,\quad N\simeq \ker(p^\prime)\simeq \ker(\tilde{p})\simeq \mathcal{M}_D,
    \]
    as desired.
\end{proof}

We will present one more lemma that will also be useful later.

\begin{lem}\label{lem:identify_free_summand}
    Let $K$ be a smoothly slice knot such that there exists a splitting $\widehat{CFD}(S^3 \setminus K) \simeq M\oplus N$ satisfying $M\simeq \widehat{CFD}(S^3 \setminus U)$, where $U$ denotes an unknot. Consider the induced splitting (via the pairing formula \cite[Theorem 11.19]{lipshitz2018bordered})
    \[
    \begin{split}
        \widehat{HFK}(S^3,K) &\simeq H_\ast (\widehat{CFA}(T_\infty,\nu)\boxtimes \widehat{CFD}(S^3 \setminus K)) \\
        &\simeq H_\ast (\widehat{CFA}(T_\infty,\nu) \boxtimes M)\oplus H_\ast (\widehat{CFA}(T_\infty,\nu) \boxtimes N) \\
        &\simeq \mathbb{F}_2 \oplus H_\ast (\widehat{CFA}(T_\infty,\nu) \boxtimes N);
    \end{split}
    \]
    note that, in the last line, we used the pairing formula again to write 
    \[
    H_\ast (\widehat{CFA}(T_\infty,\nu) \boxtimes N) \simeq \widehat{HFK}(S^3,U) \simeq \mathbb{F}_2.
    \]
    Then the generator of the $\mathbb{F}_2$ summand admits a lift to a homology class of a cycle in $CFK_\mathcal{R}(S^3,K)$ which generates its direct summand isomorphic to $\mathcal{R}$.
\end{lem}

\begin{lem} \label{lem:locally_acyclic_evenhat}
    Let $Z$ be a finitely generated free chain complex over $\mathcal{R}$ whose $U$-localization is acyclic. Then $H_\ast(\hat{Z})$ is an even-dimensional vector space.
\end{lem}
\begin{proof}
    Denote $Z_U = Z\otimes_{\mathcal{R}} \mathcal{R}/(V)$, which is a chain complex over $\mathbb{F}_2[U]$. Then the $U$-localization of $Z$ is $U^{-1}Z_U$, which is acyclic by assumption. Since $Z$ is finitely generated, $Z_U$ is also finitely generated, and thus we know that $H_\ast(Z_U)$ is the direct sum of torsion modules of the form $\mathbb{F}_2[U]/(U^n)$ for various integers $n>0$. By \cite[Lemma 4.4]{dai2019involutive}, since $Z_U$ is free, we know that $Z_U$ is homotopy equivalent to the direct sum of complexes of the form
    \[
    (\mathbb{F}_2[U]\xrightarrow{U^n\cdot \mathrm{id}}\mathbb{F}_2[U])
    \]
    for the same set of integers $n$. For each of those complexes, the homology of its hat-flavored truncation is clearly 2-dimensional. Therefore $H_\ast(\hat{Z})$ should be even-dimensional.
\end{proof}
\begin{proof}[Proof of \Cref{lem:identify_free_summand}]
    Consider the projection endomorphism $p$ of $\widehat{CFD}(S^3 \setminus K)$ which is identity on $M$ and zero on $N$. Then, by \Cref{thm:cited_thm2_projectioncorrespondence}, there exists a projection $\Lambda(p)$ on $CFK_\mathcal{R}(S^3,K)$, unique up to homotopy. Denote its image and kernel by $C$ and $D$, respectively. Then, by \Cref{thm:cited_thm3_hatflavor_recovery}, we know that $H_\ast(\hat{C})$ is 1-dimensional (over $\mathbb{F}_2$).

    We claim that $C$ is homotopy equivalent to $\mathcal{R}$ with zero differential. Since $K$ is smoothly slice, $CFK_\mathcal{R}(S^3,K)$ admits a splitting of the form $\mathcal{R}\oplus Z$, where the $U$-localization of $Z$ is acyclic. Then, by \cite[Corollary 1.2]{popovic2023link}, we have two cases:
    \begin{itemize}
        \item either $C\simeq \mathcal{R}\oplus Z$ for some $Z$ with acyclic $U$-localization,
        \item or $C$ itself has acyclic $U$-localization. 
    \end{itemize}
    If the latter case is true, then by \Cref{lem:locally_acyclic_evenhat}, $H_\ast(\hat{C})$ should be even-dimensional, a contradiction. Hence the former case is true, and moreover, since we have
    \[
    H_\ast(\hat{Z})\oplus \mathbb{F}_2 \simeq H_\ast(\hat{C}) \simeq \mathbb{F}_2,
    \]
    $\hat{Z}$ is acyclic (over $\mathbb{F}_2$), and thus we can apply \Cref{lem: hat acyclic implies acyclic} to deduce that $Z$ itself is acyclic (over $\mathcal{R}$). Thus we get
    \[
    C\simeq \mathcal{R}\oplus Z\simeq \mathcal{R}.
    \]
    Denote the generator of its homology by $\tilde{c}$. By \Cref{thm:cited_thm3_hatflavor_recovery}, we know that the hat-flavored truncation of $C$ is exactly the given $\mathbb{F}_2$-summand. Therefore $\tilde{c}$ is a lift of $c$; since $\tilde{c}$ generates the homology of an $\mathcal{R}$-summand of $CFK_\mathcal{R}(S^3,K)$, the lemma is proven.
\end{proof}

\section{An ansatz towards a proof of \Cref{thm:main}} \label{sec:ansatz}

Recall that a triple $(Y,W,f)$ is said to be a \emph{cork} if
\begin{itemize}
    \item $Y$ is a homology 3-sphere,
    \item  $W$ is a contractible 4-manifold with an identification $\partial W = Y$,
    \item $f:Y\rightarrow Y$ is a diffomorphism which does not extend smoothly to $W$.
\end{itemize}
Given a knot $K$ (in $S^3$) and a slice disk $D$ (in $B^4$) bounding $K$, we can consider the $(+1)$-surgery $B^4 _{+1}(D)$ of $B^4$ along the disk $D$, defined as follows. Choose a point $p$ in the interior of $D$. Removing a small ball neighborhood $N(p)$ from $B^4$ (and also $D$) gives a concordance $C=D\setminus N(p)$, inside $S^3 \times I = B^4 \setminus N(p)$, from the unknot to $K$. Choose a tubular neighborhood $N(C)\simeq D^2 \times S^1 \times I$ of $C$. Then we can perform a $(+1)$-surgery along $C$, by removing $N(C)$ from $S^3 \times I$ and gluing back along the $+1$ slope. This operation produces a homology cobordism between $S^3 _{+1}(K)$ and $S^3$; we then cap off the $S^3$ boundary by attaching a standard 4-ball to it. It is clear that the diffeomorphism class (rel boundary) of $B^4 _{+1}(D)$ depends only on the smooth isotopy class (rel boundary) of $D$. Furthermore, since $B^4_{+1}(D)$ is a homology ball that is simply-connected, it is always a contractible manifold, bounding the homology sphere $S^3 _{+1}(K)$.

To produce a cork, it suffices to construct a homology sphere $Y$, together with a pair of contractible 4-manifolds bounded by $Y$ which are diffeomorphic as smooth manifolds with boundary but not diffeomorphic rel boundary. Such manifolds can be constructed as follows. Given a knot $K$, suppose that we are given a diffeomorphism $f: S^3 \rightarrow S^3$ which fixes $K$ pointwise. Such a diffeomorphism induces a \emph{deform-spun disk} $D_{K,f}$ which bounds $K\sharp -K$, which we will define in \Cref{def: deform spun} below.

\begin{defn}\label{def:spinning}
Let $a$ be a properly embedded smooth arc in $D^3$.
  Furthermore, let $\phi \colon I \times D^3 \to D^3$ be an isotopy of $D^3$ such that
  $\phi_0 = \mathrm{id}_{D^3}$, $\phi_t \vert_{\partial D^3} = \mathrm{id}_{\partial D^3}$ for every $t \in I$, and $\phi_1(a) = a$.
  Then the \emph{deform-spun} slice disk $D_{a, \phi} \subset D^4$ is defined by taking
  \[
  \bigcup_{t \in I} \{t\} \times \phi_t(a) \subset I \times D^3,
  \]
  and rounding the corners along $\{0, 1\} \times \partial D^3$.
  When the arc $a$ is understood, we simply write $D_{\phi}$ instead of $D_{a, \phi}$.
\end{defn}

It was observed in \cite[Lemma 3.3]{juhasz2018stabilization} that given an orientation-preserving self-diffeomorphism $f$ of $(D^3, a)$ such that $f|_{\partial D^3} = \mathrm{id}_{\partial D^3}$, there exists an isotopy $\phi \colon I \times D^3 \to D^3$, such that $\phi_1 = f$. Furthermore, the isotopy class of the deform-spun disk $D_{a, \phi}$ only depends on $f$. Hence we will denote $D_{a,\phi}$ by $D_{a, f}$ for simplicity.

\begin{defn}\label{def: deform spun}
Let $K$ be a knot in $S^3$, and suppose that $B$ is an open 3-ball that intersects $K$ in an unknotted arc.
Then $(S^3 \smallsetminus B, K \smallsetminus B)$ is diffeomorphic to a ball-arc pair $(D^3, a)$.
Suppose that we are given a diffeomorphism $f \in \mathrm{Diff}(S^3, K)$ that is the identity on $B$.
Then the \emph{deform-spun} slice disk $D_{K, f} \subset B^4$ for $-K \# K$ is defined
to be $D_{a, f|_{S^3 \smallsetminus B}}$.
\end{defn}

Note that, when $f$ is not smoothly isotopic to the identity map, the induced deform-spun disk $D_{K,f}$ is in general not smoothly isotopic to the standard ribbon disk $D_{K,\text{id}}$. Then we have two contractible 4-manifolds $B^4 _{+1}(D_{K,\text{id}})$ and $B^4 _{+1}(D_{K,f})$ bounding $S^3 _{+1}(K\sharp -K)$. Since $D_{K,f}$ is always diffeomorphic to $D_{K,\mathrm{id}}$ when we allow nontrivial action on the boundary 3-sphere (see \cite[Proposition 3.2]{juhasz2020distinguishing} for details), $B^4 _{+1}(D_{K,\text{id}})$ and $B^4 _{+1}(D_{K,f})$ are also diffeomorphic (where we do not fix the boundary), and thus our construction defines a cork of the form $(S^3 _{+1}(K\sharp -K),B^4 _{+1}(D_{K,\text{id}}),F)$ for some diffeomorphism $F$ induced by the choice of $f$.

For the purpose of proving \Cref{thm:main}, we will consider the following setting, which will be recalled in \Cref{sec:proof}. We consider the knot $K$ of the form $K=K_0 \sharp K_0$, where $K_0$ is defined as 
\[
K_0 = (2T_{4,5} \sharp -T_{4,9})_{3,-1}.
\]
Then we consider the diffeomorphism $f=f_2 \circ f_1$ of $S^3$, fixing $K$ pointwise, where $f_1$ and $f_2$ are defined as follows. The first diffeomorphism $f_1$ maps the first $K_0$ summand of $K$ to the second summand, and vice versa, as shown in \Cref{fig:knotK}. The second map $f_2$ is defined as the ``half Dehn twist'' diffeomorphism (see also \cite[Section 1.2]{juhasz2018stabilization}), which acts as identity outside a tubular neighborhood $\nu(K)$ of $K$ and acts on $K$ as a half rotation. As discussed above, deform-spinning $K$ along $f$ and performing a $(+1)$-surgery along it defines a cork bounding $S^3 _{+1}(K\sharp -K)$. Proving that this cork survives a stabilization (i.e. performing a connected sum with a copy of $S^2 \times S^2$) is equivalent to showing that $B^4 _{+1}(D_{K,\text{id}}) \sharp (S^2 \times S^2)$ and $B^4 _{+1}(D_{K,f}) \sharp (S^2 \times S^2)$ are not diffeomorphic rel boundary.

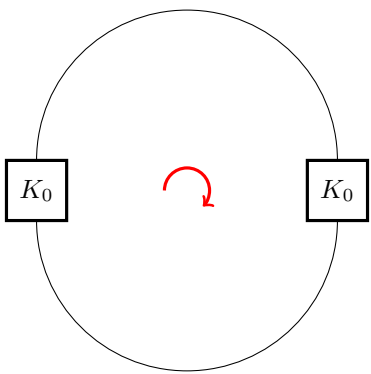
\begin{figure}[h]
\begin{tikzpicture}[squarednode/.style={rectangle, draw=black, very thick, minimum size=8mm}]

\draw (-2,0.4) arc (180:0:2);
\draw (-2,-0.4) arc (180:360:2);
\node[squarednode] (node2) at (2,0) {$K_0$};
\node[squarednode] (node4) at (-2,0) {$K_0$};
\draw[red,very thick,->] (-0.3,0) arc (180:-45:0.3);

\end{tikzpicture}
\caption{The knot $K$. The definition of $f$ starts with rotating along the center by 180 degrees.}
\label{fig:knotK}
\end{figure}

\section{Obstruction from involutive Heegaard Floer homology}\label{sec:obstruction}

Recall that, given two 3-manifolds $Y_1,Y_2$ with basepoints $z_1,z_2$, respectively, a 4-dimensional cobordism $W$ between them, a self-conjugate $\mathrm{Spin}^c$-structure $\mathfrak{s}$ on $W$, and a smooth path $\gamma$ on $W$ from $z_1$ to $z_2$, one can associate to $(W,\mathfrak{s},\gamma)$ an $\mathbb{F}_2[U,Q]/(Q^2)$-linear chain map
\[
F^I_{W,\mathfrak{s}}:CFI^-(Y_1,\mathfrak{s}\vert_{Y_1})\rightarrow CFI^-(Y_2,\mathfrak{s}\vert_{Y_2}),
\]
whose homotopy class depends only on the smooth isotopy class of $(W,\mathfrak{s},\gamma)$. When $W$ is a simply-connected 4-manifold bounding a homology sphere $Y$ (which is naturally considered as a cobordism from $S^3$ to $Y$), then it carries a unique $\mathrm{Spin}^c$ structure, and any two possible choices of $\gamma$ are smoothly isotopic, so we may drop $\mathfrak{s}$ and $\gamma$ from our notation and write
\[
F^I_{W}:CFI^-(S^3)\rightarrow CFI^-(Y).
\]
Clearly, if two simply-connected 4-manifolds $W_1$ and $W_2$ bounding $Y$ satisfy $F^I_{W_1}\neq F^I_{W_2}$, then they are not diffeomorphic rel boundary. This fact will be used to develop an obstruction for a cork to stay exotic after one stabilization.

\begin{lem} \label{lem:stabilizedmap}
Let $W$ be a spin 4-manifold bounding a homology sphere $Y$. Consider the induced (non-involutive) cobordism map 
\[
F^-_W:CF^-(S^3)\rightarrow CF^-(Y).
\]
Then the involutive cobordism map
\[
F^I_{W\sharp (S^2 \times S^2)}:CFI^-(S^3)\rightarrow CFI^-(Y)
\]
induced by $W\sharp (S^2 \times S^2)$ is given by $F^I_{W\sharp (S^2 \times S^2)}=QF^-_W$.
\end{lem}
\begin{proof}
Consider the involutive cobordism map $F^I_W$, induced by $W$, and write 
\[
F^I_W = G+QH.
\]
Since truncating involutive Heegaard Floer theory by $Q=0$ recovers (non-involutive) Heegaard Floer theory, we should have 
\[
G = F^I_W\vert_{Q=0} \sim F^-_W.
\]
Now, since the involutive cobordism map induced by $S^2 \times S^2$ is homotopic to the multiplication map by $Q$ \cite[Theorem 13.1]{hendricks2022naturality}, we see that 
\[
F^I_{W\sharp (S^2 \times S^2)} \sim F^I_W \circ F^I_{S^2 \times S^2} \sim QF^I_{W} = Q(G+QH) = QG \sim QF^-_{W}, 
\]
as desired.
\end{proof}

\begin{lem} \label{lem:4dimobs}
Let $W_1,W_2$ be homology 4-balls bounding a homology 3-sphere $Y$. If $W_1\sharp (S^2 \times S^2)$ and $W_2 \sharp (S^2 \times S^2)$ are diffeomorphic rel boundary, then $c_{W_1} + c_{W_2}$ (as an element of $HF^-(Y)$) is contained in the image of the action of $1+\iota_Y$ on $HF^-(Y)$.
\end{lem}
\begin{proof}
For each $i=1,2$, the involutive cobordism map
\[
F^I_{W_i\sharp (S^2 \times S^2)} : CFI^-(S^3)\rightarrow CFI^-(Y)
\]
induced by $W_i\sharp (S^2 \times S^2)$ is given by $QF^-_{W_i}$, where $F^-_{W_i}$ denotes the cobordism map
\[
F^-_{W_i} : CF^-(S^3)\rightarrow CF^-(Y)
\]
induced by $W_i$ on the ordinary minus-flavored Heegaard Floer chain complex. Hence, if $W_1\sharp (S^2 \times S^2)$ and $W_2 \sharp (S^2 \times S^2)$ are diffeomorphic rel boundary, then it follows from \Cref{lem:stabilizedmap} that
\[
Q(F^-_{W_1} + F^-_{W_2}) \sim 0.
\]
Let $\tilde{H}=G+QH$ be a nullhomotopy of $Q(F^-_{W_1} + F^-_{W_2})$. Then we have 
\[
\begin{split}
    Q(F^-_{W_1} + F^-_{W_2}) &= \partial_I \tilde{H} + \tilde{H} \partial_I \\
    &= (\partial + Q(1+\iota))(G+QH) + (G+QH)(\partial + Q(1+\iota)) \\
    &= \partial G + G\partial + Q(\iota_Y G + G\iota_Y + \partial H + H \partial), 
\end{split}
\]
where $\partial_I$ denotes the differential on $CFI^-(Y)$. Thus we see that 
\[
\partial G + G \partial = 0,\,F^-_{W_1} + F^-_{W_2} = \iota_Y G + G \iota_Y + H\partial + \partial H.
\]
In other words, $G$ is a chain map and $F^-_{W_1} + F^-_{W_2}$ is homotopic to $\iota_Y G + G \iota_Y$. Hence, if we denote the homology class of the cycle $G(1)$ by $c_G$, then we deduce that
\[
c_{W_1} + c_{W_2} = [F^-_{W_1}(1) + F^-_{W_2}(1)] = [(\iota_Y G + G \iota_Y)(1)] = (1+\iota_Y)(c_G).
\]
Therefore $c_{W_1} + c_{W_2}$ is contained in the image of $1+\iota_Y$.
\end{proof}

\Cref{lem:4dimobs} already gives us an obstruction for homology 4-balls with the same boundary from being smoothly diffeomorphic rel boundary after one stabilization. However, dealing directly with Heegaard Floer homology of $(+1)$-surgeries is not easy, so we will use large surgery formula to reduce our obstruction to a more easily computable one.

\begin{lem}\label{lem:cfkmapsurgery}
Let $C$ be a concordance between knots $K_1$ and $K_2$. For any integer $n>0$, consider the $n$-surgery $(S^3 \times I)_n(C)$ along $C$, which is a homology cobordism from $S^3 _n(K_1)$ to $S^3 _n(K_2)$. Then, when $n$ is sufficiently large, the following diagram commutes.
\[
\xymatrix{
HF^-(S^3 _n (K_1),[0]) \ar[rr]^(0.4){\Gamma_{n,0}}\ar[dd]_{F^-_{(S^3 \times I)_n(C)}} && A_0(K_1) \subset CFK_{UV}(S^3,K_1) \ar@<20pt>[dd]^{F_C} \\
\\
HF^-(S^3 _n (K_2),[0]) \ar[rr]^(0.4){\Gamma_{n,0}} && A_0(K_2) \subset CFK_{UV}(S^3,K_2)
}
\]
Here, $[0]$ denotes the (self-conjugate) zero $\mathrm{Spin}^c$ structure on $n$-surgeries along $K_1$ and $K_2$, $F_{(S^3 \times I)_n(C)}$ denotes the minus-flavored Heegaard Floer cobordism map induced by $(S^3 \times I)_n(C)$, and $F_C$ denotes the cobordism map on knot Floer homology, induced by $C$, endowed with a suitable decoration.
\end{lem}
\begin{proof}
The large surgery isomorphism can be described as a link cobordism map in the following way. Consider the 2-handle cobordism from $S^3$ to $S^3 _{+n}(K)$ and turn it upside down so that it goes from the $(+n)$-surgery to $S^3$. Then the core of this 2-handle is a smoothly embedded disk $\Sigma_{K,n}$ from the empty link in $S^3 _{+n}(K)$ to $K$ in $S^3$. The map $\Gamma_{n,s}$ can then be described as the cobordism map induced by this map, where the ambient 4-manifold is endowed with the $\mathrm{Spin}^c$-structure given by $s$. This cobordism can also be seen as an ambient 0-surgery cobordism as in \Cref{fig:handlediagram}. Note that this fact was used implicitly in \cite[Section 3]{hendricks2022involutivedual}.

Instead of using the projection map $S^3 \times I\rightarrow I$ as our Morse function, we may choose a different one so that $C$ is a straight cylinder and 1,2,3-handles are attached to its complement. Hence this cobordism commutes up to diffeomorphism rel boundary with the 0-surgery cobordism used for the large surgery isomorphism, and also the $(+n)$-surgery cobordism on the component $K_{disk}$ which corresponds to both $U$ and $K$ (before and after the 1,2,3-handle attachments). Therefore, the link cobordisms which induce the composition of the top and right maps and the composition of the left and bottom maps in the diagram are diffeomorphic rel boundary, and thus the lemma follows.
\end{proof}

\begin{figure}
\centering
\includegraphics[height=4cm]{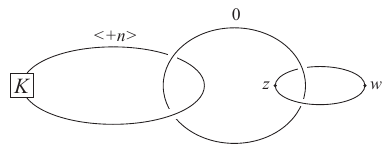}
\caption{A relative Kirby diagram representing the core of a 2-handle cobordism from $S^3 _{+n}(K)$ to $(S^3,K)$. The component corresponding to the knot $K$ is the one with the basepoints $z$ and $w$ drawn.}
\label{fig:handlediagram}
\end{figure}

As a result, we have the following proposition, which now deals only with knot Floer homology.
\begin{prop}\label{prop:obstructionprop}
Let $K$ be a knot and $D_1,D_2$ be two slice disks for $K$. Suppose that $B^4 _{+1}(D_1) \sharp (S^2 \times S^2)$ and $B^4_{+1}(D_2) \sharp (S^2 \times S^2)$ are diffeomorphic rel boundary. Then the classes $t_{D_1},t_{D_2}\in \widehat{HFK}(S^3,K)$ (see \Cref{rem:invarianttD} for their definition) satisfy the condition
\[
t_{D_1} + t_{D_2} \in \mathrm{Im}(1+\iota_K).
\]
\end{prop}
\begin{proof}
Suppose on the contrary that $t_{D_1} + t_{D_2} \not\in \mathrm{Im}(1+\iota_K)$. By applying the arguments of \cite[Section 4]{dai20222} to \Cref{lem:cfkmapsurgery}, it follows that there exists a degree-preserving map $f$ such that the following diagram commutes up to homotopy, for each $i=1,2$.
\[
\xymatrix{
HF^-(S^3 _{+1} (U)) \ar[rr]^(0.4){=}\ar[dd]_{F^-_{B^4 _{+1}(D_i)}} && A_0(U) \subset CFK_{UV}(S^3,U) \ar@<15pt>[dd]^{F_{D_i}} \\
\\
HF^-(S^3 _{+1} (K)) \ar[rr]^(0.4){f} && A_0(K) \subset CFK_{UV}(S^3,K)
}
\]
Furthermore, $f$ commutes with the involutions, i.e. $\iota_K \circ f \sim f \circ \iota_{S^3 _{+1}(K)}$. Since $t_{D_i}$ is defined to be the hat-flavored truncation of the image of the generator $1\in \mathbb{F}_2[U,V]\simeq CFK_{UV}(S^3,U)$ under the map $F_{D_i}$, it follows from our assumption that
\[
F^-_{B^4 _{+1}(D_1)}(1) + F^-_{B^4 _{+1}(D_2)}(1) \not\in \mathrm{Im}(1+\iota_{S^3 _{+1}(K)}).
\]
It now follows from \Cref{lem:4dimobs} that $B^4 _{+1}(D_1) \sharp (S^2 \times S^2)$ and $B^4_{+1}(D_2) \sharp (S^2 \times S^2)$ are not diffeomorphic rel boundary.
\end{proof}

Before we proceed to the next section, we will use the tools that we have developed to formulate an explicitly computable obstruction that can be used for our examples. 
\begin{defn} \label{defn:involutively_weird}
    A knot $K$ is \emph{involutively weird} if $\widehat{HFK}(S^3,K)$ admits an $\iota_K$-invariant splitting
    \[
    \widehat{HFK}(S^3,K) \simeq V_1 \oplus V_2 \oplus W,
    \]
    together with some splitting
    \[
    CFK^-(S^3 \setminus K) \simeq C_1 \oplus D, 
    \]
    such that the following conditions are satisfied.
    \begin{itemize}
        \item $V_1$ and $V_2$ are 1-dimensional;
        \item The sub-splitting $\widehat{HFK}(S^3,K) \simeq V_1 \oplus (V_2 \oplus W)$ is invariant under the (hat-flavored) basepoint actions $\Phi$ and $\Psi$, defined in \cite[Section 3]{zemke2017quasistabilization};
        \item $V_1$ and $V_2 \oplus W$ are the homology of the hat-flavored truncation of $C_1$ and $D$, respectively;
        \item There exists a direct summand of the form $(a\xrightarrow{U} b)$ in $D$ (i.e. a summand generated by $a$ and $b$ satisfying $\partial a = Ub$ and $\partial b = 0$) such that the homology class of $a$ in $\widehat{CFK}(S^3,K)$ generates $V_2$.
    \end{itemize}
\end{defn}
\begin{prop}\label{prop: involutive_weird_implies_result}
    Let $K_0$ be an involutively weird knot. Then $S^3 _{+1}(2K_0\sharp -2K_0)$ admits a pair of contractible 4-manifolds that stay exotic rel boundary after one stabilization.
\end{prop}
\begin{proof}
    Recall that $K=K_0 \sharp K_0$. It follows from \cite[Theorem 5.1]{juhasz2020distinguishing} that, under the identification 
\[
\widehat{HFK}(S^3,K\sharp -K)\simeq \mathrm{Hom}(\widehat{HFK}(S^3,K),\widehat{HFK}(S^3,K)),
\]
the element $t_{D_{K,f}}$ corresponds to the induces action of $f$ on $HFK(S^3,K)$, whereas $t_{D_{K,\mathrm{id}}}$ corresponds to the identity map. It also follows from \cite[Theorem 8.1]{juhasz2018stabilization} that the induced action of $f$, under the identification
\[
\widehat{HFK}(S^3,K)\simeq \widehat{HFK}(S^3,K_0)\otimes \widehat{HFK}(S^3,K_0)
\]
is given by 
\[
f_\ast = \mathrm{Sw}\circ (\mathrm{id}+g).
\]
Here, $\mathrm{Sw}$ is the isomorphism defined by swapping the two copies of $\widehat{HFK}(S^3,K_0)$ with the second copy, and
\[
g = 1\otimes 1 + \Phi\otimes \Psi.
\]
In other words, the element $t_{D_{K,\mathrm{id}}}+t_{D_{K,f}}$ can be written as 
\[
\sum_{x} \sum_{x} x^\ast \otimes y^\ast \otimes (y\otimes x + \Psi(y)\otimes \Phi(x)),
\]
where $x,y$ runs over basis elements of $\widehat{CFK}(S^3,K_0)$. 

Since $K_0$ is involutively weird, it admits a splitting
    \[
    \widehat{HFK}(S^3,K) \simeq V_1 \oplus V_2 \oplus W
    \]
    where the conditions in \Cref{defn:involutively_weird} are satisfied. Choose bases of $V_1$, $V_2$, and $W$, (write the generators of $V_1$ and $V_2$ as $v_1$ and $v_2$, respectively) and write their union as $\mathcal{B}$, so that the elements $a,b$ in \Cref{defn:involutively_weird} are contained in $\mathcal{B}$; we write the set of canonical duals of elements of $\mathcal{B}$ by $\mathcal{B}^\ast$. We consider the simple tensor 
    \[
    s = v^\ast_2 \otimes v^\ast_1 \otimes v_1 \otimes v_2
    \]
    of elements in $\mathcal{B}$ and $\mathcal{B}^\ast$.
    
    We claim that, when we express $t_{D_{K,\mathrm{id}}}+t_{D_{K,f}}$ as a sum of simple tensors of elements in $\mathcal{B}$ and $\mathcal{B}^\ast$, the term $s$ has coefficient 1. Suppose the claim is false. Then the element
\[
\sum_{x} \sum_{x} x^\ast \otimes y^\ast \otimes \Psi(y)\otimes \Phi(x),
\]
when expressed as a sum of simple tensors of elements in $\mathcal{B}$ and $\mathcal{B}^\ast$, should have $s$ with coefficient 1. This means that, for some $x,y\in\mathcal{B}$, the element
\[
x^\ast \otimes y^\ast \otimes \Psi(y) \otimes \Phi(x)
\]
also has $s$ with coefficient 1. This can only happen if $x^\ast = v^\ast_2$ and $y^\ast = v^\ast_1$, i.e. $x=v_2$ and $y=v_1$. But then $\Psi(y)=\Psi(v_1)=0$, a contradiction. The claim is thus proven.

    We now claim that $t_{D_{K,\mathrm{id}}}+t_{D_{K,f}}$ is not contained in the image of $1+\iota_K$ in $\widehat{HFK}(S^3,K)$. Since $K=K_0\sharp K_0 \sharp -K_0 \sharp -K_0$, the action of $\iota_K$ can be computed by the involutive connected sum formula \cite[Theorem 1.1]{zemke2019connected}. In particular, by applying the formula 3 times, we get
    \[
    \begin{split}
        \iota_{2K_0 \sharp -2K_0} &\sim (1 + \Psi^\ast_{K_0} \otimes \Phi_{-K_0 \sharp 2K_0}) \circ (\iota^\ast_{K_0}\otimes \iota_{-K_0 \sharp 2K_0}) \\
        &\sim \cdots \\
        &\sim (1+\mathcal{T}) \circ (\iota^\ast_{K_0}\otimes \iota^\ast_{K_0} \otimes \iota_{K_0} \otimes \iota_{K_0}),
    \end{split}
    \]
    where $\mathcal{T}$ is defined as 
    \[
    \begin{split}
        \mathcal{T} &= \Psi^\ast \otimes \Phi^\ast \otimes 1 \otimes 1+\Psi^\ast \otimes 1 \otimes \Phi \otimes 1 + \Psi^\ast \otimes 1 \otimes 1 \otimes \Phi + 1 \otimes \Psi^\ast \otimes \Phi \otimes 1 \\
        &\quad + 1 \otimes \Psi^\ast \otimes 1 \otimes \Phi + 1 \otimes 1 \otimes \Psi \otimes \Phi + \Psi^\ast \otimes \Phi^\ast \otimes \Psi \otimes \Phi + 1 \otimes \Psi^\ast \otimes \Psi\Phi \otimes \Phi \\
        &\quad + \Psi^\ast \otimes 1 \otimes \Psi\Phi \otimes \Phi + \Psi^\ast \otimes \Psi^\ast\Phi^\ast \otimes \Phi \otimes 1+\Psi^\ast \otimes \Psi^\ast\Phi^\ast \otimes 1 \otimes \Phi + \Psi^\ast \otimes \Psi^\ast\Phi^\ast \otimes \Psi\Phi \otimes \Phi.
    \end{split}
    \]
    To prove the claim, we only have to prove the following statement: for any simple tensor $x=a^\ast\otimes b^\ast \otimes c\otimes d$ of elements $a,b,c,d\in \mathcal{B}$, when we express $s+\iota_K(s)$ as a sum of simple tensors of elements in $\mathcal{B}$ and $\mathcal{B}^\ast$, the coefficient of $s$ is zero. Suppose that this statement is false. Since we have 
    \[
    \iota_{K_0}(v_1)=v_1,\quad \iota_{K_0}(v_2)=v_2
    \]
    by the definition of involutive weirdness, we have to show that the coefficient of $s$ in $\mathcal{T}(x)$ is one. Among the terms in $\mathcal{T}$, the ones that may have nonzero coefficient of $s$ must have the form
    \[
    (\text{something})\otimes 1 \otimes 1 \otimes (\text{something}),
    \]
    since by the definition of involutive weirdness, $\Phi(v_1)=\Psi(v_1)=0$ and there is no element in $\mathcal{B}$ whose image under either $\Phi$ or $\Psi$, when expressed as a sum of elements in $\mathcal{B}$, can contain $v_1$. The only terms satisfying this condition is $\Psi^\ast \otimes 1 \otimes 1 \otimes \Phi$, which means that the coefficient of $s(=v^\ast_2 \otimes v^\ast_1 \otimes v_1 \otimes v_2)$ in
    \[
    \Psi^\ast(a^\ast)\otimes b^\ast \otimes c \otimes \Phi(d)
    \]
    should be 1. However, since the action of $\Phi$ on the hat-flavored knot Floer homology simply acts by counting terms in $CFK_\mathcal{R}$ with coefficient exactly $U$, it follows from the definition of involutively weird knots that there is no element $z$ in $\widehat{HFK}(S^3,K_0)$ such that the coefficient of $v_2$ in $\Phi(z)$, when expressed as a sum of elements in $\mathcal{B}$, so this situation cannot happen. Hence the claim is proven.

    Therefore, using \Cref{prop:obstructionprop}, it follows from the claim that the contractible 4-manifolds $W_1$ and $W_2$, formed by taking $(+1)$-surgery along the deform-spun disks $D_{K,\mathrm{id}}$ and $D_{K,f}$, are not diffeomorphic rel boundary after connected-summing with $S^2 \times S^2$. The proposition follows.
\end{proof}

For a future usage, especially in \cite{guth2024invariant}, we will also develop a simpler criterion that implies the ``one is not enough'' result. However, this criterion will not be used in this paper; we will stick to \Cref{defn:involutively_weird} and \Cref{prop: involutive_weird_implies_result}.

\begin{defn}
    A knot $K$ is \emph{involutively weird} if $\widehat{HFK}(S^3,K)$ admits a splitting
    \[
    \widehat{HFK}(S^3,K) \simeq V_1 \oplus V_2 \oplus W
    \]
    which is invariant under the actions of $\iota_K$, $\Phi$, and $\Psi$, and satisfies $\dim(V_1)=\dim(V_2)=1$.
\end{defn}
\begin{prop}
    Let $K$ be an involutively weird knot. Then $S^3 _{+1}(4K_0\sharp -4K_0)$ admits a pair of contractible 4-manifolds that stay exotic rel boundary after one stabilization.
\end{prop}
\begin{proof}
    Following the proof of \Cref{prop: involutive_weird_implies_result} gives this proposition as well; we use the monomial
    \[
    s = v^\ast_1 \otimes v^\ast_2 \otimes v^\ast_2 \otimes v^\ast_1 \otimes v_2 \otimes v_1 \otimes v_1 \otimes v_2
    \]
    and take the 180 degree rotation action on the connected sum of four copies of $K_0$. Since we assumed that the splitting $V_1 \oplus V_2 \oplus W$ is also invariant under $\Psi$, the proof in this case is much easier. The most crucial part of the proof would be to prove that for any simple tensor of the form
    \[
    x = a^\ast \otimes b^\ast \otimes c^\ast \otimes d^\ast \otimes e \otimes f \otimes g \otimes h,
    \]
    the coefficient of $s$ in $\mathcal{T}(x)$ is always zero, where $\mathcal{T}$ satisfies
    \[
    \iota_{4K_0 \sharp -4K_0} \sim (1+\mathcal{T})\circ (\iota^\ast_{K_0} \otimes \iota^\ast_{K_0} \otimes \iota^\ast_{K_0} \otimes \iota^\ast_{K_0} \otimes \iota_{K_0} \otimes \iota_{K_0} \otimes \iota_{K_0} \otimes \iota_{K_0}).
    \]
    As in the proof of \Cref{prop: involutive_weird_implies_result}, $\mathcal{T}$ is a sum of simple tensors involving $1$, $\Phi$, $\Psi$, $\Phi^\ast$, and $\Psi^\ast$, where each term contains at least one of either $\Phi$ or $\Phi^\ast$ and one of either $\Psi$ or $\Psi^\ast$. It then follows from the definition of involutive nontriviality that the coefficient of $s$ in $\mathcal{T}(x)$ is always zero.
\end{proof}

\section{Proof of the main theorem} \label{sec:proof}
To prove \Cref{thm:main}, we have to compute the action of $\iota_K$ on $CFK_{UV}(S^3,K_0)$; recall that the knot $K_0$ is defined as
\[
K_0 = (-T_{4,9} \sharp 2T_{4,5})_{3,-1}.
\]
Of course, we will not be able to compute the total action of $\iota_K$. However we are able to determine a small part of the action, and it will turn out to be sufficient for our purposes. We assume that the reader is familiar with standard notations and techniques in bordered Heegaard Floer theory, as in \cite{lipshitz2018bordered}.

For simplicity, write the knot $-T_{4,9} \sharp 2T_{4,5}$ as $K_1$, so that $K_0 = (K_1)_{3,-1}$. It is shown in \cite[Remark 4.8]{hendricks2021quotient} that the $CFK_{UV}(S^3,K_1)$ is $\iota_K$-locally equivalent to the following complex $C$, where $x$ and $d$ have the same bidegree, $(0,0)$.
\[
\xymatrix{
 && b \ar[dd]_{V^2} && a \ar[ll]_{U^2} \ar[dd]^{V^2} \\
x & \oplus & &&\\
 && d && c \ar[ll]^{U^2}
}
\]
Here, the $\iota_K$-action is given as follows.
\[
a\mapsto a,\,x\mapsto x+d,\,b\mapsto c,\,c\mapsto b,\,d\mapsto d.
\]
By \cite[Remark 4.8]{hendricks2020surgery}, we have a decomposition
\[
CFK_{UV}(S^3,K_1) \simeq C \oplus D
\]
of $\iota_K$-complexes, where $D$ is a direct sum of rectangular complexes. This also implies that $K_1$ is $\mathcal{R}$-multirectangular. Since $H_\ast(\hat{C})$ is clearly 5-dimensional, we can apply \Cref{prop:keyprop} to show that there exists an $\iota_{S^3 \setminus K_1}$-invariant splitting 
\[
\widehat{CFD}(S^3 \setminus K_1)\simeq M\oplus N
\]
such that $M\simeq \mathcal{M}_C$ and $N\simeq \mathcal{M}_D$. This discussion will be used in the proof of \Cref{lem:the_knot_is_weird}.

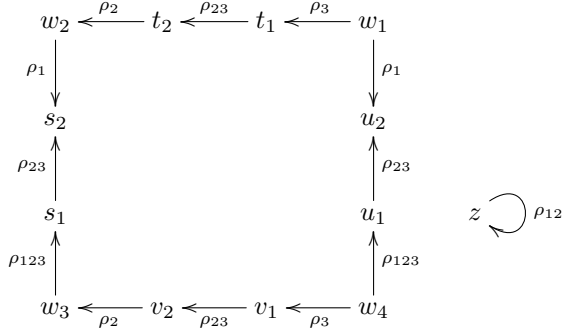
\begin{figure}
\[
\xymatrix{
w_2 \ar[d]_{\rho_1} & t_2 \ar[l]_{\rho_2} & t_1 \ar[l]_{\rho_{23}} & w_1 \ar[l]_{\rho_3} \ar[d]^{\rho_1} \\
s_2 & & & u_2 \\
s_1 \ar[u]^{\rho_{23}} & & & u_1 \ar[u]_{\rho_{23}} & z \ar@(ur,dr)^{\rho_{12}} \\
w_3 \ar[u]^{\rho_{123}} & v_2 \ar[l]^{\rho_2} & v_1 \ar[l]^{\rho_{23}} & w_4 \ar[l]^{\rho_3} \ar[u]_{\rho_{123}}
}
\]
\caption{The type-D structure $M$, a direct summand of $\widehat{CFD}(S^3 \setminus K_1)$.}
\label{fig:componentM}
\end{figure}

Furthermore, if we denote the $(3,-1)$-cabling pattern inside a solid torus as $P_{3,-1}$, then the type-A structure $CFA^-(T_\infty,P_{3,-1})$ can be described as in \Cref{fig:cablingmodule}. Note that we denote its truncation by $U=0$ as $\widehat{CFA}(T_\infty,P_{3,-1})$.

\begin{figure}
\[
\xymatrix{
& a_3 \ar[ld]_{\rho_2}\ar[dd]^{U^3} && a_2\ar[ll]_{\rho_2,\rho_1}\ar[dd]^{U^2} && a_1\ar[ll]_{\rho_2,\rho_1} \ar[dd]^U\\
c \ar[rd]_{\rho_3} \\
& b_3 && b_2\ar[ll]^{U\cdot (\rho_2,\rho_1)} && b_1\ar[ll]^{U\cdot (\rho_2,\rho_1)}
}
\]
\caption{The type-A structure $CFA^-(T_\infty,P_{3,-1})$. Its truncation by $U=0$, which corresponds to removing all arrows whose label contains $U$, gives $\widehat{CFA}(T_{\infty},P_{3,-1})$.}
\label{fig:cablingmodule}
\end{figure}
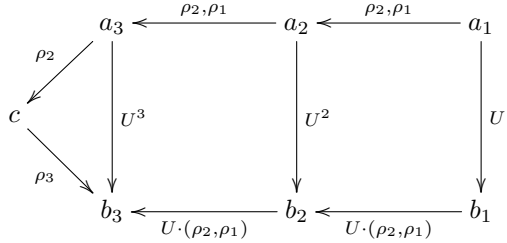

We can now carry out a partial computation of the $\iota_K$-action on the knot Floer chain complex of $K_0 = (K_1)_{3,-1}$. Recall that we chose a decomposition of $\widehat{CFD}(S^3 \setminus K_1)$ as
\[
\widehat{CFD}(S^3 \setminus K_1) \simeq M \oplus N.
\]
Tensoring this with $\widehat{CFA}(T_\infty,P_{3,-1})$ gives
\[
\widehat{CFK}(S^3,K_0) \simeq \widehat{CFA}(T_\infty,P_{3,-1}) \boxtimes \widehat{CFD}(S^3 \setminus K_1) \simeq M_{3,-1} \oplus N_{3,-1}
\]
via pairing formula, where we denote the tensor products of $M$ and $N$ with $\widehat{CFA}(T_\infty,P_{3,-1})$ by $M_{3,-1}$ and $N_{3,-1}$, respectively.

Observe that we can compute the $\mathcal{R}$-coefficient knot Floer chain complex, $CFK_{\mathcal{R}}(S^3,K_0)$, of $K_0$, from $\widehat{CFD}(S^3 \setminus K_1)$ via immersed curve cabling formula of Hanselman and Watson \cite{hanselman2019cabling}, where $\mathcal{R}$ denotes the ring $\mathbb{F}_2[U,V]/(UV)$. The component we get from the direct summand $M$, which should correspond to $M_{3,-1}$ via cabling formula, is given as the direct sum
\[
C_1 \oplus C_2 \oplus C_3 \oplus (\mathcal{R}\cdot \omega),
\]
where $\omega$ is $(0,0)$-bigraded and generates the free summand $\mathcal{R} \cdot \omega$, and the rest are given as described in \Cref{fig:Mcomplexes}. The bidegrees of their generators are presented in \Cref{fig:bidegrees}.

It is natural to ask how the pairing formula identifies the homology class of $\widehat{HFK}$ as shown above with elements of $\widehat{CFA}(T_\infty,P_{3,-1}) \boxtimes M$. To see this, we note that $c\otimes z$ has bidegree $(0,0)$ and $a_1 \otimes v_2$ has bidegree $(1,1)$. It follows from the immersed curve computation, which we omit for simplicity, that there exists a unique hat-flavored homology classes of bidegrees $(0,0)$ (which is clearly $\omega$) and $(1,1)$ (denoted as $\zeta$ in \Cref{fig:Mcomplexes}. Hence we see that $c\otimes z$ and $a_1 \otimes v_2$ are identified with $\omega$ and $\zeta$. We will also note that $b_1 \otimes v_2$ is identified with $\alpha$ in the figure. We do not have to know about how other homology classes are identified.

\begin{figure}[h]
\[
\xymatrix@C=0.7em@R=0.7em{
  & \bullet\ar@{.>}[d] & & \bullet \ar[ll]\ar@{.>}[d] & & & & & \bullet\ar@{.>}[dddddd] & \bullet\ar[l]\ar@{.>}[dd] \\
 &\bullet & \bullet\ar[l]\ar@{.>}[dd] & \bullet & & \bullet \ar[ll]\ar@{.>}[dd] & & & & & & \bullet\ar@{.>}[d] & & & & & &\bullet \ar[llllll]\ar@{.>}[d] & \\
  & & & & & & & & & \bullet & \bullet \ar[l]\ar@{.>}[dddddd] & \bullet & & \bullet\ar[ll]\ar@{.>}[d] & & & & \bullet & & \bullet\ar[ll]\ar@{.>}[d] & \\
 & & \alpha & \zeta\ar[l]\ar@{.>}[d] & & \bullet & \bullet \ar[l]\ar@{.>}[dd] & & & & & & & \bullet & & & & & & \bullet\ar[llllll] & & \omega\\
 & & & \bullet & &\bullet\ar[ll]\ar@{.>}[d] \\
 & & & & & \bullet & \bullet \ar[l] & & &  \\
 & & & & & & & & \bullet & \bullet\ar[l]\ar@{.>}[dd] \\
& & &   \\
& & & & & & & & & \bullet & \bullet\ar[l]
}
\]
\caption{The complexes $C_1$, $C_2$, $C_3$, and $\mathcal{R}\cdot \omega$, from left to right. Horizontal arrows of length $\ell$ denotes a term in the differential with coefficient $U^\ell$. Vertical arrows should correspond to $V^\ell$ terms in the differential, but this is not known, and we don't have to consider them in our arguments.}
\label{fig:Mcomplexes}
\end{figure}
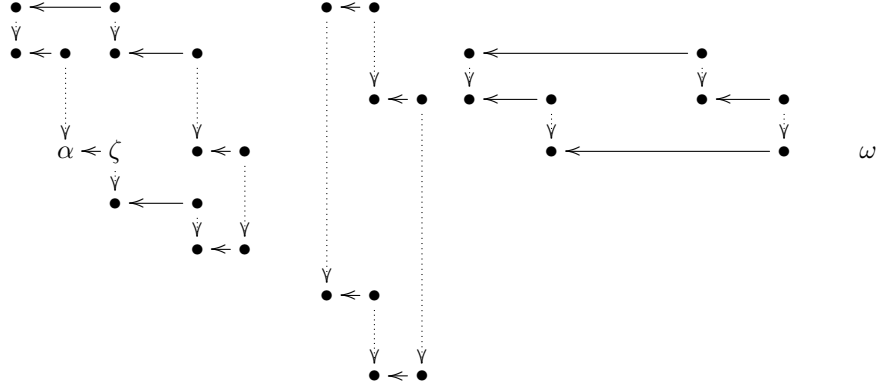

\begin{figure}[h]
\[
\xymatrix@C=0.7em@R=0.7em{
  & (5,-5)\ar@{.>}[d] & & (2,-4) \ar[ll]\ar@{.>}[d] & & & & & (1,-7)\ar@{.>}[dddddd] & (0,-6)\ar[l]\ar@{.>}[dd] \\
 &(4,-4) & (3,-3)\ar[l]\ar@{.>}[dd] & (1,-3) & & (-2,-2) \ar[ll]\ar@{.>}[dd] \\
  & & & & & & & & & (-1,-3)& (-2,-2) \ar[l]\ar@{.>}[dddddd] \\
 & & (2,0) & (1,1)\ar[l]\ar@{.>}[d] & & (-3,1) & (-4,2) \ar[l]\ar@{.>}[dd] \\
 & & & (0,2) & &(-3,3)\ar[ll]\ar@{.>}[d] \\
 & & & & & (-4,4) & (-5,5) \ar[l] & & &  \\
 & & & & & & & & (0,4) & (-1,5) \ar[l]\ar@{.>}[dd] \\
& & &   \\
& & & & & & & & & (-2,8) & (-3,9)\ar[l] \\
(9,-3)\ar@{.>}[d] & & & & & & (-2,-2)\ar@{.>}[d]\ar[llllll] \\
(8,-2) & & (5,-1)\ar[ll]\ar@{.>}[d] & & & & (-3,-1) & & (-6,0)\ar[ll]\ar@{.>}[d] & & (0,0) \\
& & (4,0) & & & & & & (-7,1) \ar[llllll]
}
\]
\caption{The $(n_z,n_w)$ bidegrees for the complexes $C_1,C_2,C_3$, and $\omega$.}
\label{fig:bidegrees}
\end{figure}
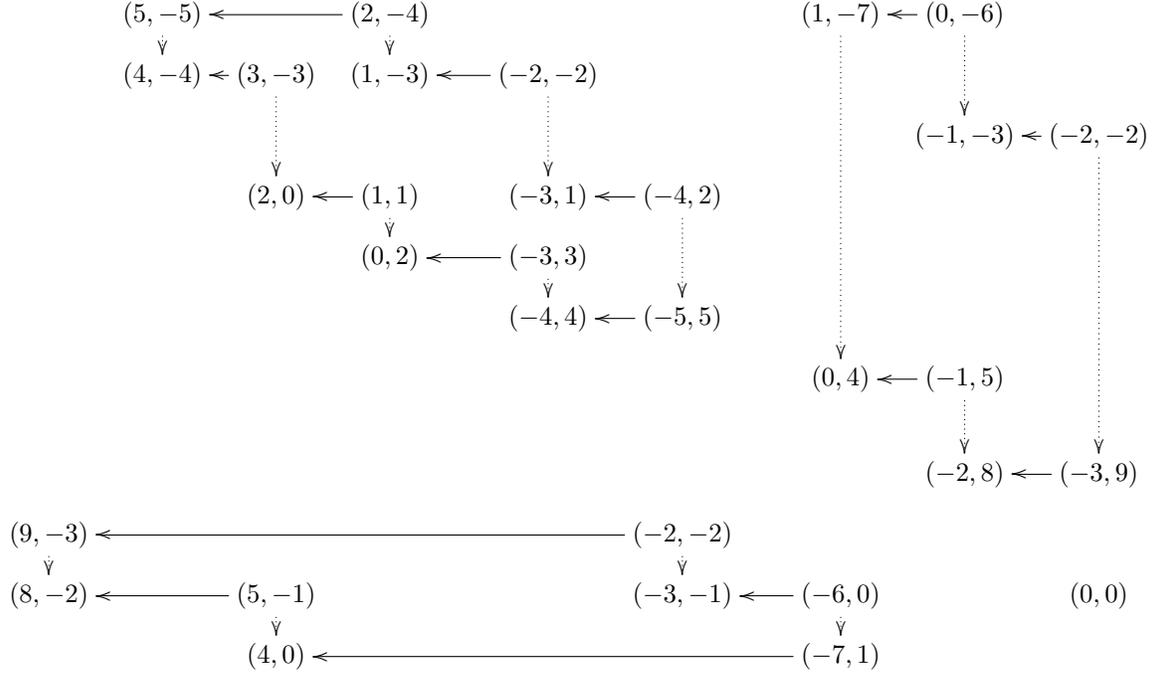

From this data, we can prove that $K_0$ is involutively weird. We will not prove that it is involutively nontrivial, as we do not need that stronger condition to prove our main result. However, one can use the results of \cite{guth2024invariant} to show that it is in fact involutively nontrivial.
\begin{lem} \label{lem:the_knot_is_weird}
    The knot $K_0 = (-T_{4,9}\sharp 2T_{4,5})_{3,-1}$ is involutively weird.
\end{lem}
\begin{proof}
    We start by observing that $M$ has a direct summand homotopy equivalent to $\widehat{CFD}(S^3 \setminus U)$, generated by $z$ in \Cref{fig:componentM}; we denote it $M_0$, so that we have a splitting
    \[
    M \simeq M_0 \oplus M_1.
    \]
    Then we consider the splitting
    \[
    CFK^-(S^3,K_0) \simeq (CFA^-(T_\infty,\nu)\boxtimes M_0) \oplus (CFA^-(T_\infty,\nu)\boxtimes (M_1 \oplus N)).
    \]
    The summand $CFA^-(T_\infty,\nu)\boxtimes (M_1 \oplus N)$ then admits direct summand generated by $\alpha$ and $\zeta$, i.e.
    \[
    CFA^-(T_\infty,\nu)\boxtimes (M_1 \oplus N) \simeq (\zeta\xrightarrow{U}\alpha)\oplus (\text{the rest}).
    \]
    The element $\zeta$ then becomes a cycle after hat-flavored truncation. 
    
    We then consider the summand $M_0$. It is clear that $\widehat{CFA}(T_\infty,\nu)\boxtimes M_0$ has 1-dimensional homology, generated by $\omega$ in \Cref{fig:componentM}. By \Cref{lem:identify_free_summand}, we see that $\omega$ is the hat-flavored truncation of a cycle in $CFK_\mathcal{R}(S^3,K_0)$ which generates its direct summand isomorphic to $\mathcal{R}$. Since the basepoint actions can be computed as the partial derivatives of the knot Floer differential with respect to the variables $U$ and $C$ \cite[Section 3]{zemke2017quasistabilization}, and we are only considering their actions on the hat flavored knot Floer homology, it follows that the splitting
    \[
    \widehat{CFK}(S^3,K_0) \simeq (\widehat{CFA}(T_\infty,\nu)\boxtimes M_0) \oplus (\widehat{CFA}(T_\infty,\nu)\boxtimes (M_1 \oplus N))
    \]
    is invariant under $\hat\Phi$ and $\hat\Psi$. 

    Hence, in order to show that $K_0$ is involutively weird, we only have to show that the splitting
    \[
    \widehat{HFK}(S^3,K_0)\simeq \mathbb{F}_2\cdot \omega \oplus \mathbb{F}_2 \cdot \zeta \oplus (\text{the rest})
    \]
    is $\iota_K$-invariant. To prove this, we first use the following fact in a previous discussion: there exists an $\iota_{S^3 \setminus K_1}$-invariant splitting
    \[
    \widehat{CFD}(S^3 \setminus K_1) \simeq M \oplus N.
    \]
    For simplicity, we will denote the restrictions of $\iota_{S^3 \setminus K_1}$ to $M$, which is a homotopy equivalence from $\widehat{CFDA}(\mathbf{AZ})\boxtimes M$ to $M$, as $\iota_M$. By \cite[Theorem 1.3]{kang2022involutive}, we see that the following splitting is $\iota_K$-invariant:
    \[
    \widehat{CFK}(S^3,K_0) \simeq (\widehat{CFA}(T_\infty,\nu)\boxtimes M) \oplus (\widehat{CFA}(T_\infty,\nu)\boxtimes N),
    \]
    thus we only have to show that the summand $\widehat{CFA}(T_\infty,\nu)\boxtimes M$ $\iota_K$-invariantly splits off $\omega$ and $\zeta$. 

    To prove this, we simply consider the ($(n_z,n_w)$-)bidegree of generators of $\widehat{CFA}(T_\infty,\nu)\boxtimes M$, as shown in \Cref{fig:bidegrees}. We see that $\zeta$ is the only generator with bidegree $(1,1)$, and $\omega$ is the only generator with bidegree $(0,0)$. Since $\iota_K$ maps elements of bidegree $(a,b)$ to elements of bidegree $(b,a)$, it is clear that $\omega$ and $\zeta$ splits of $\iota_K$-invariantly from $\widehat{CFA}(T_\infty,\nu)\boxtimes M$. Therefore $K_0$ is involutively weird.
\end{proof}

Now we are able to prove \Cref{thm:main}.
\begin{proof}[Proof of \Cref{thm:main}]
This follows from \Cref{lem:the_knot_is_weird} and \Cref{prop: involutive_weird_implies_result}.
\end{proof}

Finally, we will prove \Cref{cor:main} using \Cref{thm:main}.

\begin{proof}[Proof of \Cref{cor:main}]
By \Cref{thm:main}, we know that there exists a cork $(W,Y,f)$ such that $f$ does not extend smoothly to $W\sharp (S^2 \times S^2)$. By following the arguments used in the proof of \cite[Theorem A]{akbulut2016absolutely}, one can construct a homology cobordism $X$ from $Y$ to another homology sphere $N$, admitting a left inverse $\bar{X}$, i.e. $\bar{X}\cup_N X \simeq Y\times I$, such that any self-diffeomorphism of $N$ extends to a self-diffeomorphism of $X$ which acts by identity on $Y$. Then we consider the 4-manifolds
\[
V = W \cup X,\quad V^\prime = W \cup_{f} X,
\]
which are simply-connected homology balls by \cite[Proposition 2.6]{akbulut2016absolutely}, so that they are contractible, hence homeomorphic. Then one can simply follow the remaining part of the proof of \cite[Theorem A]{akbulut2016absolutely} to conclude that there exists no diffeomorphism between $V\sharp (S^2 \times S^2)$ and $V^\prime \sharp (S^2 \times S^2)$.
\end{proof}

\bibliographystyle{amsalpha}
\bibliography{ref}

\newcommand{\etalchar}[1]{$^{#1}$}
\providecommand{\bysame}{\leavevmode\hbox to3em{\hrulefill}\thinspace}
\providecommand{\MR}{\relax\ifhmode\unskip\space\fi MR }
\providecommand{\MRhref}[2]{%
  \href{http://www.ams.org/mathscinet-getitem?mr=#1}{#2}
}
\providecommand{\href}[2]{#2}
\begin{thebibliography}{DKM{\etalchar{+}}22}

\bibitem[AR16]{akbulut2016absolutely}
Selman Akbulut and Daniel Ruberman, \emph{Absolutely exotic compact
  4-manifolds}, Comment. Math. Helv. \textbf{91} (2016), no.~1, 1--19.
  \MR{3471934}

\bibitem[Coh23]{cohen2023composition}
Jesse Cohen, \emph{Composition maps in {H}eegaard {F}loer homology}, arXiv
  preprint arXiv:2301.08882 (2023).

\bibitem[DKM{\etalchar{+}}22]{dai20222}
Irving Dai, Sungkyung Kang, Abhishek Mallick, JungHwan Park, and Matthew
  Stoffregen, \emph{The $(2, 1) $-cable of the figure-eight knot is not
  smoothly slice}, arXiv preprint arXiv:2207.14187 (2022).

\bibitem[DM19]{dai2019involutive}
Irving Dai and Ciprian Manolescu, \emph{Involutive {H}eegaard {F}loer homology
  and plumbed three-manifolds}, Journal of the Institute of Mathematics of
  Jussieu \textbf{18} (2019), no.~6, 1115--1155.

\bibitem[Fre82]{freedman1982topology}
Michael~Hartley Freedman, \emph{The topology of four-dimensional manifolds}, J.
  Differential Geometry \textbf{17} (1982), no.~3, 357--453. \MR{679066}

\bibitem[GK24]{guth2024invariant}
Gary Guth and Sungkyung Kang, \emph{Invariant splitting principles for the
  {L}ipshitz--{O}zsv{\'a}th--{T}hurston correspondence}, arXiv preprint
  arXiv:2404.06618 (2024).

\bibitem[Gom84]{gompf1984stable}
Robert~E. Gompf, \emph{Stable diffeomorphism of compact {$4$}-manifolds},
  Topology Appl. \textbf{18} (1984), no.~2-3, 115--120. \MR{769285}

\bibitem[HHSZ20]{hendricks2020surgery}
Kristen Hendricks, Jennifer Hom, Matthew Stoffregen, and Ian Zemke,
  \emph{Surgery exact triangles in involutive {H}eegaard {F}loer homology},
  arXiv preprint arXiv:2011.00113 (2020).

\bibitem[HHSZ22a]{hendricks2022involutivedual}
\bysame, \emph{An involutive dual knot surgery formula}, arXiv preprint
  arXiv:2205.12798 (2022).

\bibitem[HHSZ22b]{hendricks2022naturality}
\bysame, \emph{Naturality and functoriality in involutive {H}eegaard {F}loer
  homology}, arXiv preprint arXiv:2201.12906 (2022).

\bibitem[HHSZ22c]{hendricks2021quotient}
\bysame, \emph{On the quotient of the homology cobordism group by {S}eifert
  spaces}, Trans. Amer. Math. Soc. Ser. B \textbf{9} (2022), 757--781.
  \MR{4480068}

\bibitem[HL19]{hendricks2019involutivebordered}
Kristen Hendricks and Robert Lipshitz, \emph{Involutive bordered {F}loer
  homology}, Trans. Amer. Math. Soc. \textbf{372} (2019), no.~1, 389--424.
  \MR{3968773}

\bibitem[HM17]{hendricks2017involutive}
Kristen Hendricks and Ciprian Manolescu, \emph{Involutive {H}eegaard {F}loer
  homology}, Duke Math. J. \textbf{166} (2017), no.~7, 1211--1299. \MR{3649355}

\bibitem[HW19]{hanselman2019cabling}
Jonathan Hanselman and Liam Watson, \emph{Cabling in terms of immersed curves},
  arXiv preprint arXiv:1908.04397 (2019).

\bibitem[JM16]{juhaszmarengon}
Andr\'{a}s Juh\'{a}sz and Marco Marengon, \emph{Concordance maps in knot
  {F}loer homology}, Geom. Topol. \textbf{20} (2016), no.~6, 3623--3673.
  \MR{3590358}

\bibitem[JTZ21]{juhasz2021naturality}
Andr\'{a}s Juh\'{a}sz, Dylan~P. Thurston, and Ian Zemke, \emph{Naturality and
  mapping class groups in {H}eegard {F}loer homology}, Mem. Amer. Math. Soc.
  \textbf{273} (2021), no.~1338, v+174. \MR{4337438}

\bibitem[JZ18]{juhasz2018stabilization}
Andr\'{a}s Juh\'{a}sz and Ian Zemke, \emph{Stabilization distance bounds from
  link {F}loer homology}, arXiv preprint arXiv:1810.09158 (2018).

\bibitem[JZ20]{juhasz2020distinguishing}
\bysame, \emph{Distinguishing slice disks using knot {F}loer homology}, Selecta
  Math. (N.S.) \textbf{26} (2020), no.~1, Paper No. 5, 18. \MR{4045151}

\bibitem[Kan22]{kang2022involutive}
Sungkyung Kang, \emph{Involutive knot {F}loer homology and bordered modules},
  arXiv preprint arXiv:2202.12500 (2022).

\bibitem[LOT11]{lipshitz2011heegaard}
Robert Lipshitz, Peter~S. Ozsv\'{a}th, and Dylan~P. Thurston, \emph{Heegaard
  {F}loer homology as morphism spaces}, Quantum Topol. \textbf{2} (2011),
  no.~4, 381--449. \MR{2844535}

\bibitem[LOT18]{lipshitz2018bordered}
\bysame, \emph{Bordered {H}eegaard {F}loer homology}, Mem. Amer. Math. Soc.
  \textbf{254} (2018), no.~1216, viii+279. \MR{3827056}

\bibitem[Man79]{mandelbaum1979decomposing}
Richard Mandelbaum, \emph{Decomposing analytic surfaces}, Geometric topology
  ({P}roc. {G}eorgia {T}opology {C}onf., {A}thens, {G}a., 1977), Academic
  Press, New York-London, 1979, pp.~147--217. \MR{537731}

\bibitem[MS21]{highercorks}
Paul Melvin and Hannah Schwartz, \emph{Higher order corks}, Invent. Math.
  \textbf{224} (2021), no.~1, 291--313. \MR{4228504}

\bibitem[OS04]{ozsvath2004holomorphic}
Peter~S. Ozsv\'{a}th and Zolt\'{a}n Szab\'{o}, \emph{Holomorphic disks and
  topological invariants for closed three-manifolds}, Ann. of Math. (2)
  \textbf{159} (2004), no.~3, 1027--1158. \MR{2113019}

\bibitem[OS06]{ozsvath2006holomorphic}
\bysame, \emph{Holomorphic triangles and invariants for smooth four-manifolds},
  Adv. Math. \textbf{202} (2006), no.~2, 326--400. \MR{2222356}

\bibitem[OS08]{ozsvath2008knot}
\bysame, \emph{Knot {F}loer homology and integer surgeries}, Algebr. Geom.
  Topol. \textbf{8} (2008), no.~1, 101--153. \MR{2377279}

\bibitem[Pop23]{popovic2023link}
David Popovi{\'c}, \emph{Link {F}loer homology splits into snake complexes and
  local systems}, arXiv preprint arXiv:2306.05546 (2023).

\bibitem[Sar15]{sarkar2015moving}
Sucharit Sarkar, \emph{Moving basepoints and the induced automorphisms of link
  {F}loer homology}, Algebr. Geom. Topol. \textbf{15} (2015), no.~5,
  2479--2515. \MR{3426686}

\bibitem[Wal64]{wall}
C.~T.~C. Wall, \emph{On simply-connected {$4$}-manifolds}, J. London Math. Soc.
  \textbf{39} (1964), 141--149. \MR{163324}

\bibitem[Zem17]{zemke2017quasistabilization}
Ian Zemke, \emph{Quasistabilization and basepoint moving maps in link {F}loer
  homology}, Algebr. Geom. Topol. \textbf{17} (2017), no.~6, 3461--3518.
  \MR{3709653}

\bibitem[Zem19a]{zemke2019connected}
\bysame, \emph{Connected sums and involutive knot {F}loer homology}, Proc.
  Lond. Math. Soc. (3) \textbf{119} (2019), no.~1, 214--265. \MR{3957835}

\bibitem[Zem19b]{zemke2019link}
\bysame, \emph{Link cobordisms and functoriality in link {F}loer homology}, J.
  Topol. \textbf{12} (2019), no.~1, 94--220. \MR{3905679}

\end{thebibliography}
\end{document}